\documentclass[a4paper, 12pt]{article} 

\usepackage[utf8]{inputenc}
\usepackage[T1]{fontenc}
\usepackage[a4paper, margin=2.75cm]{geometry}
\usepackage{needspace, mathscinet}
\usepackage{blkarray} 
\usepackage{multirow}
\usepackage{subcaption}
\usepackage{libertine} 
\usepackage{inconsolata} 

\usepackage{authblk, amsmath, amssymb, amsthm, graphicx, mathtools}
\usepackage[textsize=small, textwidth=2cm, color=yellow]{todonotes}
\usepackage[colorlinks=true, citecolor=blue, linkcolor=blue, urlcolor=blue, hypertexnames=false]{hyperref}
\usepackage[inline]{enumitem}
\usepackage{thmtools,thm-restate}

\usepackage[noabbrev,capitalize,nameinlink]{cleveref}

\hypersetup{    
    pdftitle={The Erdős-Pósa property for circle graphs as vertex-minors},
    pdfauthor={Rutger Campbell, J.~Pascal Gollin, Meike Hatzel, O-joung Kwon, Rose McCarty, Sang-il Oum, and Sebastian Wiederrecht}
}

\usepackage{macros}

\def\cqedsymbol{\ifmmode$\lrcorner$\else{\unskip\nobreak\hfil
\penalty50\hskip1em\null\nobreak\hfil$\lrcorner$
\parfillskip=0pt\finalhyphendemerits=0\endgraf}\fi}

\let\le\leqslant
\let\ge\geqslant
\let\leq\leqslant
\let\geq\geqslant

\let\OLDthebibliography\thebibliography
\renewcommand\thebibliography[1]{
  \OLDthebibliography{#1}
  \setlength{\parskip}{0pt}
  \setlength{\itemsep}{0pt plus 0.3ex}
}%
\AtBeginDocument{
   \def\MR#1{}
}%

\newcommand\mail[1]{\href{mailto:#1}{\texttt{#1}}}

\title{The Erd\H{o}s-P\'{o}sa property for circle graphs as~vertex-minors}

\author{Rutger Campbell\thanks{Supported by the Institute for Basic Science (IBS-R029-C1).}}
\affil{Discrete Mathematics Group, Institute for Basic Science (IBS), Daejeon, South Korea.}
\author{J.~Pascal Gollin\thanks{Supported in part by the Slovenian Research and Innovation Agency (research project N1-0370) and in part by the Institute for Basic Science (IBS-R029-Y3).}}
\affil{FAMNIT, University of Primorska, Koper, Slovenia.}
\author[1]{Meike Hatzel\textsuperscript{\textasteriskcentered}}
\author[3,1]{O-joung Kwon\textsuperscript{\textasteriskcentered}\thanks{Supported by the National Research Foundation of Korea (NRF) grant funded by the Ministry of Science and ICT (No.~RS-2023-00211670)}}
\affil{Department of Mathematics, Hanyang University, Seoul, South Korea.}
\author{Rose McCarty\thanks{Supported by the National Science Foundation under Grant No.~DMS-2202961.}}
\affil{School of Mathematics and School of Computer Science, Georgia Institute of Technology, USA. }
\author[1]{Sang-il Oum\textsuperscript{\textasteriskcentered}}
\author{Sebastian Wiederrecht\textsuperscript{\textasteriskcentered}}
\affil{School of Computing, KAIST, Daejeon, South Korea.}
\affil[ ]{\small E-mail: 
\mail{rutger@ibs.re.kr},
\mail{pascal.gollin@famnit.upr.si},
\mail{research@meikehatzel.com},
\mail{ojoungkwon@hanyang.ac.kr},
\mail{rmccarty3@gatech.edu},
\mail{sangil@ibs.re.kr},
\mail{wiederrecht@kaist.ac.kr}
}

\begin{document}

\maketitle

\begin{abstract}
We prove that for any circle graph $H$ with at least one edge and for any positive integer $k$, there exists an integer {$t=t(k,H)$} so that every graph $G$ either has a vertex-minor isomorphic to the disjoint union of $k$ copies of $H$, or has a $t$-perturbation with no vertex-minor isomorphic to $H$. Using the same techniques, we also prove that for any planar multigraph $H$, every binary matroid either has a minor isomorphic to the cycle matroid of $kH$, or is a low-rank perturbation of a binary matroid with no minor isomorphic to the cycle matroid of $H$.
\end{abstract}

\section{Introduction}

In 1965, Erd\H{o}s and P\'{o}sa~\cite{EP1965} proved that every graph contains $k$ vertex-disjoint cycles or a vertex set of size at most $\mathcal{O}(k \log k)$ such that after its removal, the remaining graph is acyclic.
Based on this, a graph $H$ is said to have the \emph{Erd\H{o}s-P\'{o}sa property} if each graph~$G$ contains $k$ vertex-disjoint copies of $H$ (a \textsl{packing}), or has a set of size bounded by a function of $k$ (a \textsl{hitting set}) such that removing it from $G$ results in a graph not containing~$H$.
Here, one can be flexible with the notion of what it means for one graph to \textsl{contain} another.
The question has been asked for several containment relations, for example, minors~\cite{simonovits1967,graphMinors5}, immersions~\cite{liu2015}, and topological minors~\cite{thomassen1988}. Other objects like directed graphs~\cite{reed1996,packingdirectedcircuits1996}, and other settings, such as coarse graph theory~\cite{EPfarCycles, EPlongCycles}, have also been considered. 
We refer the reader to the dynamic survey by Raymond~\cite{EPwebpageRaymond} for an overview.

A result by Robertson and Seymour~\cite{graphMinors5} says that planar graphs are precisely the graphs that have the Erd\H{o}s-P\'{o}sa property with respect to minors. This result has received a good deal of attention~\cite{tighterEPplanar}, and now we even know efficient bounds that match the original theorem of Erd\H{o}s and P\'{o}sa~\cite{tightEPplanar}.
In this paper, we consider the containment relation of \emph{vertex-minors} and prove an Erd\H{o}s-P\'{o}sa property for \emph{circle graphs}.

Roughly, the \emph{vertex-minors} of a graph $G$ are the graphs that can be obtained from $G$ by deleting vertices and by performing certain ``local moves'' called local complementations; \emph{locally complementing} at a vertex $v$ replaces the induced subgraph on the neighborhood of $v$ by its complement. Many interesting graph properties are invariant under local complementation, such as the interlace polynomial~\cite{interlaceVM}, the level of entanglement of an associated quantum state~\cite{VMCliffordTrans}, and the cut-rank function~\cite{RWAndVM}, which gives a notion of graph connectivity inspired by rank in the adjacency matrix.

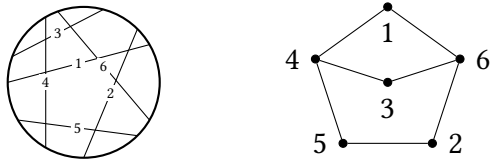
\begin{figure}
    \centering
    \begin{tikzpicture}
        \tikzstyle{v}=[circle, draw, solid, fill=black, inner sep=0pt, minimum width=3pt]
        \tikzstyle{w}=[solid, fill=white, inner sep=1pt]
        \draw [thick](0,0) circle (1);
        \draw (30:1)--(180:1) node [pos=0.5,w]{\tiny 1};
        \draw (45:1)--(-90:1) node [pos=0.5,w]{\tiny 2};
        \draw (75:1)--(160:1) node [pos=0.5,w]{\tiny 3};
        \draw (-45:1)--(210:1) node [pos=0.5,w]{\tiny 4} node [pos=0.5,w]{\tiny 5};
        \draw (110:1)--(-30:1) node [pos=0.5,w]{\tiny 6};
        \draw (120:1)--(-120:1) node [pos=0.5,w]{\tiny 4};
        \begin{scope}[xshift=4cm]
            \tikzstyle{v}=[circle, draw, solid, fill=black, inner sep=0pt, minimum width=3pt]
            \node [v,label=below:1] (v1) at (90:1) {};
            \node [v,label=left:4] (v4) at (90+72:1){};
            \node [v,label=left:5] (v5) at (90+72*2:1){};
            \node [v,label=right:2] (v2) at (90+72*3:1){};
            \node [v,label=right:6] (v6) at (90+72*4:1){};
            \node [v,label=below:3] (v3) at (0,0){};
            \draw (v1)--(v4)--(v5)--(v2)--(v6)--(v1);
            \draw (v4)--(v3)--(v6);
        \end{scope}
    \end{tikzpicture}
    \caption{A chord diagram and its corresponding circle graph.}
    \label{fig:circle}
\end{figure}

A graph is a \emph{circle graph} if there exists a collection of chords on a circle so that the vertices are in bijection with the chords, and two vertices are adjacent if and only if their chords have non-empty intersection; see \cref{fig:circle}. 
There is a connection between planar graphs/circle graphs and minors/vertex-minors, respectively. 
Essentially, there is a common generalization of the two settings via isotropic systems. This connection was discovered by Bouchet~\cite{isoSystems, graphicIsoSystems} and de Fraysseix~\cite{VMandInterlacement}. We refer the reader to~\cite{circlePivoting}.

For vertex-minors, we need a different notion that is analogous to hitting sets.
Given an integer $t$, we say that a graph $G_1$ is a \emph{$t$-perturbation} of a graph $G_2$ if they have the same vertex set, and there exists a graph on $\abs{V(G_1)}+t$ vertices that contains both $G_1$ and $G_2$ as vertex-minors.
For convenience, given a graph $H$ and an integer $k$, we write $kH$ for the graph that is the disjoint union of $k$ copies of $H$.
We prove the following theorem.

\begin{restatable}{theorem}{maintheorem}\label{thm:main}
For any circle graph $H$ with at least one edge and for any positive integer $k$, there exists an integer $t = t(k, H)$ so that every graph~$G$ either has a vertex-minor isomorphic to~$kH$, or has a $t$-perturbation $\widetilde{G}$ which has no vertex-minor isomorphic to~$H$.
\end{restatable}

\noindent We note that if $H$ is a graph with no edges, then any graph $G$ with no vertex-minor isomorphic to~$kH$ has a bounded number of vertices in terms of $k$ and $\abs{V(H)}$. To see this, note that $kH$ is an edgeless graph with $k\abs{V(H)}$ vertices. By Ramsey's Theorem, every sufficiently large graph has an independent set of size $k\abs{V(H)}$ or a clique of size $k\abs{V(H)}+1$. Locally complementing at a vertex of the clique yields a star, which contains an independent set of size $k\abs{V(H)}$. So in both cases, we obtain $kH$ as a vertex-minor.

Perturbations are the correct analog of hitting sets because of the following rough converse of \cref{thm:main}, which holds with different quantitative bounds.

\begin{restatable}{proposition}{mainProp}\label{prop:converse}
    Let $H$ be a graph and $t$ be a non-negative integer.
    If a graph~$G$ has a $t$\nobreakdash-perturbation~$\widetilde{G}$ with no vertex-minor isomorphic to $H$, then $G$ has no vertex-minor isomorphic to~$(t+1)H$.
\end{restatable}
\noindent Moreover, we prove in \cref{lem:pertNecessary} that non-trivial perturbations are necessary, that is, \cref{thm:main} does not always hold with~taking a $0$-perturbation (i.e., locally complementing) and removing $t$ vertices. 

We note that 
a graph $\widetilde{G}$ is a small perturbation of a graph $G$ if and only if we can obtain a graph that is locally equivalent to $\widetilde{G}$ by adding (over the binary field) a low-rank matrix to the adjacency matrix of $G$ and setting the diagonal to zero (see \cref{lem:rank,lem:rankConv}). This equivalent definition has been considered before in relation to vertex-minors; see Geelen's conjecture about the structure of graphs with a forbidden vertex-minor~\cite{McCartyThesis}. It is known that perturbations of classes with a forbidden vertex-minor also forbid a vertex-minor; see~\cite[Lemma~1.6.7]{McCartyThesis}. This operation of adding a low-rank matrix is also a form of ``atomic transduction'', which is an important tool for studying first-order properties of graph classes; see~\cite{logicSurvey} for a survey of recent trends in graphs and logic. In this area, the operation is usually called ``flipping''; see~\cite{flipperGames, flipWidth} for some of its many uses.



We do not know whether the condition that $H$ is a circle graph in \cref{thm:main} is necessary, although there are some reasons to guess that it might be. First of all, non-planar graphs $H$ do \emph{not} have the Erd\H{o}s-P\'{o}sa property for minors~\cite{graphMinors5}. Secondly, our proof of \cref{thm:main} relies on the Grid Theorem for Vertex-Minors~\cite{gridThmVM}. This theorem says that for any fixed circle graph $H$, every graph without $H$ as a vertex-minor has bounded rank-width. 
(We also observe that for any circle graph $H$ and positive integer $k$, the graph~$kH$ is a circle graph.) This theorem is the only place we require $H$ to be a circle graph; once we know that the host graph $G$ has bounded rank-width, we do not need any requirement on~$H$. However, for the Grid Theorem for Vertex-Minors, it \emph{is} necessary for $H$ to be a circle graph, since circle graphs themselves can have arbitrarily large rank-width. 

So it is natural to pose the following problem.

\begin{problem}
Does any non-circle graph have the Erd\H{o}s-P\'{o}sa property for vertex-minors? 
\end{problem}
\noindent That is, is there any graph $H$ which is not a circle graph but does satisfy \cref{thm:main} for every positive integer~$k$? We actually suspect that the answer is yes, but if so, this shows an interesting distinction between minors and vertex-minors.

Using the same techniques, we also prove an Erd\H{o}s-P\'{o}sa property for circle graphs as pivot-minors in bipartite graphs. This theorem (\cref{thm:mainPivot}) yields the following corollary.


\begin{restatable}{corollary}{binaryMatroids}\label{cor:binaryMatroids}
For any planar multigraph $H$ and for any positive integer~$k$, there exists an integer $p=p(k,H)$ so that each binary matroid $M$ either has a minor isomorphic to $M(kH)$, or has a rank-$p$ perturbation $\widetilde{M}$ which has no minor isomorphic to~$M(H)$.
\end{restatable}
We write $M(G)$ for the cycle matroid of a graph $G$. A binary matroid $\widetilde{M}$ is a \emph{rank-$p$ perturbation} of a binary matroid $M$ if $E(M)=E(\widetilde{M})$ and they have respective representations $\widetilde{A}$ and $A$ over the binary field so that $\widetilde{A}$ and $A$ have the same set of row indices and $\widetilde{A}-A$ has rank at most~$p$; see \cref{sec:pivots} for details. We also prove a rough converse of \cref{cor:binaryMatroids}; see \cref{prop:binMatConverse}. A generalization of \cref{cor:binaryMatroids} to $\GF(q)$-representable matroids will be proven by James Davies and the third author in an upcoming paper. For other notions of containment, Erd\H{o}s-P\'{o}sa properties have been considered for matroids; see~\cite{disjCocircuits} and its strengthening~\cite{betterDisjCocircuits}. These papers consider matroids without $k$ disjoint cocircuits. 

Rank-$p$ perturbations are a fundamental notion when studying the structure of $\GF(q)$-representable matroids~\cite{highlyConnMatroidsConj}. However, \cref{cor:binaryMatroids} does not directly compare to Robertson and Seymour's theorem that planar graphs have the Erd\H{o}s-P\'{o}sa property. There are several reasons for this discrepancy, including the fact that $kH$ is not the only graph whose cycle matroid is~$M(kH)$. Additionally, low-rank perturbations of matroids are more analogous to splitting/identifying vertices a few times, rather than to deleting a few vertices.

This paper is organized as follows. 
In \cref{sec:prelims}, we review some basic definitions about vertex-minors and rank-width and state some useful lemmas about trees.
In \cref{sec:perturbations}, we discuss perturbations and a key new idea called \emph{robustness}.
In \cref{sec:RobustParts}, we prove \cref{prop:chain}, which says that either we can find the desired perturbation, or we can find many disjoint robust parts of bounded size and very small cut-rank. 
In \cref{sec:extract}, we show how to work with the robust parts that we found in \cref{prop:chain} in order to obtain~$kH$ as a vertex-minor (see \cref{prop:cleaningChains}). 
In \cref{sec:conclusion}, we prove \cref{thm:main} and its rough converse, \cref{prop:converse}.
In \cref{sec:discussion}, we discuss properties of perturbations and prove that perturbations are necessary for \cref{thm:main} (see \cref{lem:pertNecessary}). 
In \cref{sec:pivotminor}, we discuss the relationship between pivot-minors of bipartite graphs and minors of binary matroids.
In \cref{sec:pivots}, we prove our results about pivot-minors (\cref{thm:mainPivot} and its rough converse, \cref{prop:conversePivot}). 
Finally, in \cref{sec:matroids}, we prove our results about binary matroids (\cref{cor:binaryMatroids} and its rough converse, \cref{prop:binMatConverse}).


\section{Preliminaries}
\label{sec:prelims}

In this section, we introduce some preliminaries that are used throughout the paper.
Graphs are finite and simple unless otherwise noted.
\emph{Multigraphs} are allowed to contain both loops and parallel edges. Circle graphs are the intersection graphs of chords of a circle. In other words, a graph is a \emph{circle graph} if there is a set of chords of a circle such that the vertices of the graph correspond to the chords and two vertices are adjacent if and only if the corresponding chords intersect. For a non-negative integer $k$, we let $[k]$ denote the set of all positive integers less than or equal to~$k$.

Let $G$ be a graph. We denote by $V(G)$ and $E(G)$ the vertex set and edge set of~$G$, respectively. 
For a vertex~$v$ in~$G$, we denote by $G-v$ the graph obtained from $G$ by deleting the vertex $v$ and all edges incident with~$v$. 
For a set~$S$ of vertices in $G$, we denote by $G-S$ the graph obtained from $G$ by deleting all vertices in $S$ and all edges incident with $S$. 
For an edge $e$ in $G$, we denote by $G-e$ the graph obtained from $G$ by deleting $e$. 
For a subset~$F$ of $E(G)$, we denote by $G-F$ the graph obtained from $G$ by deleting all edges in $F$. For $S\subseteq V(G)$, we let $G[S]$ denote the subgraph of~$G$ induced by $S$. For a vertex $v$ of~$G$, we let $N_G(v)$ denote the set of neighbors of $v$ in $G$. A vertex is not considered to be in its own neighborhood. We omit the subscript if $G$ is clear from context.


\subsection*{Vertex-minors and pivot-minors}

\begin{figure}
    \centering
    \begin{subfigure}{0.2\textwidth}
        \centering
        \begin{tikzpicture}
        \tikzstyle{v}=[circle, draw, solid, fill=black, inner sep=0pt, minimum width=3pt]
        \node [v,label=$v$] (v) at (-.5,.5) {};
        \node [v,label=$w$] (w) at (.5,.5){};
        \draw (v) to (w);
        \foreach \i in {1,2,3,4,5,6,7}{
            \node [v] (v\i) at (180-30+\i*30:1) {};
        }
        \draw (v)--(v1);
        \draw (v)--(v2);
        \draw (v)--(v3);
        \draw (v)--(v4);
        \draw (w)--(v3);
        \draw (w)--(v4);
        \draw (w)--(v5);
        \draw (w)--(v6);
        \draw (v1)--(v3);
        \draw (v2)--(v4); 
        \draw (v3)--(v4);
        \draw (v4)--(v6);
        \draw (v5)--(v7);
        \draw (v1)--(v7);
        \draw (v2)--(v7);
    \end{tikzpicture}
    \caption{$G$}
    \end{subfigure}
    \begin{subfigure}{0.2\textwidth}
        \centering
        \begin{tikzpicture}
        \tikzstyle{v}=[circle, draw, solid, fill=black, inner sep=0pt, minimum width=3pt]
        \tikzstyle{r}=[red]
        \node [v,label=$v$] (v) at (-.5,.5) {};
        \node [v,label=$w$] (w) at (.5,.5){};
        \draw (v) to (w);
        \foreach \i in {1,2,3,4,5,6,7}{
            \node [v] (v\i) at (180-30+\i*30:1) {};
        }
        \draw (v)--(v1);
        \draw (v)--(v2);
        \draw (v)--(v3);
        \draw (v)--(v4);
        \draw (w)--(v5);
        \draw (w)--(v6);
        \draw [r] (v1)--(v4);
        \draw [r] (v1)--(v2);
        \draw [r] (v2)--(v3);
        \draw [r] (w)--(v1);
        \draw [r] (w)--(v2);
        \draw (v4)--(v6);
        \draw (v5)--(v7);
        \draw (v1)--(v7);
        \draw (v2)--(v7);
    \end{tikzpicture}
    \caption{$G*v$}
    \end{subfigure}
    \begin{subfigure}{0.2\textwidth}
        \centering
        \begin{tikzpicture}
        \tikzstyle{v}=[circle, draw, solid, fill=black, inner sep=0pt, minimum width=3pt]
        \tikzstyle{r}=[red]
        \tikzstyle{b}=[blue]
        \node [v,label=$w$] (v) at (-.5,.5) {};
        \node [v,label=$v$] (w) at (.5,.5){};
        \draw (v) to (w);
        \foreach \i in {1,2,3,4,5,6,7}{
            \node [v] (v\i) at (180-30+\i*30:1) {};
        }
        \draw (v)--(v1);
        \draw (v)--(v2);
        \draw (v)--(v3);
        \draw (v)--(v4);
        \draw (w)--(v3);
        \draw (w)--(v4);
        \draw (w)--(v5);
        \draw (w)--(v6);
        \draw [r] (v1)--(v4);
        \draw [r] (v2)--(v3);
        \draw [r] (v3)--(v5);
        \draw [r] (v3)--(v6);
        \draw [r] (v4)--(v5);
        \draw [r] (v1)--(v5);
        \draw [r] (v1)--(v6);
        \draw [r] (v2)--(v5);
        \draw [r] (v2)--(v6);
        \draw (v3)--(v4);
        \draw (v5)--(v7);
        \draw (v1)--(v7);
        \draw (v2)--(v7);
    \end{tikzpicture}
    \caption{$G\pivot vw$}
    \end{subfigure}
    \caption{Local complementation and pivoting.}
    \label{fig:lc}
\end{figure}
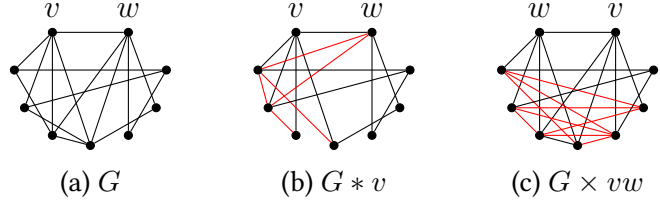

Let $G$ be a graph. For a vertex $v$ of $G$, \emph{locally complementing} at~$v$ removes all existing edges between all neighbors of~$v$ and adds edges between all non-adjacent pairs of neighbors of~$v$ in~$G$.
The resulting graph is denoted by $G*v$. 
For an edge $uv$ of~$G$, \emph{pivoting} at $uv$ firstly removes all existing edges between every pair of vertices in $V(G)\setminus\{u,v\}$ having distinct sets of neighbors in $\{u,v\}$, 
secondly adds edges between all non-adjacent pairs of vertices in $V(G)\setminus\{u,v\}$ having distinct sets of neighbors in $\{u,v\}$, and finally 
swaps the labels of~$u$ and $v$. 
The resulting graph is denoted by~$G\pivot uv$.
See \cref{fig:lc} for an illustration.
It is well known that $G \pivot uv = G*u*v*u = G*v*u*v$; see~\cite{RWAndVM}.

We say that two graphs are \emph{locally equivalent} (respectively, \emph{pivot-equivalent}) if one can be obtained from the other by a sequence of local complementations (respectively, pivots).
We say that a graph $H$ is a \emph{vertex-minor} (respectively, \emph{pivot-minor}) of a graph~$G$ if $H$ is an induced subgraph of a graph that is locally equivalent (respectively, pivot-equivalent) to~$G$.
Equivalently, $H$ is a vertex-minor of~$G$ if $H$ can be obtained from $G$ by any sequence of local complementations and vertex deletions, performed in any order. Likewise, $H$ is a pivot-minor of~$G$ if $H$ can be obtained from $G$ by any sequence of pivots and vertex deletions, performed in any order.

We require the following well-known lemma, which shows that there are only three ways to remove a vertex up to local equivalence. To explain this lemma properly, we need to introduce some notation. Given a graph $G$ and a vertex $v$ of $G$, we write $G/v$ for the graph that is obtained from $G$ by deleting $v$ if $v$ is isolated, and otherwise by selecting an arbitrary neighbor $u$ of $v$, pivoting on $uv$, and deleting $v$. This graph is well-defined up to pivot-equivalence because, for distinct neighbors $u_1$ and $u_2$ of $v$, we have $G \times vu_1 = (G \times vu_2 ) \times u_1u_2$; see~\cite[Proposition~2.5]{RWAndVM}.

A direct proof of the first part of the lemma can be found in~\cite[Lemma~1.4]{circlePivoting}, while the second part of the lemma follows from~\cite[Lemma~3.1]{circlePivoting}.

\begin{lemma}[{Bouchet~\cite[(9.2)]{graphicIsoSystems}} and Fon-Der-Flaass~{\cite[Corollary 4.3]{FonDerFlaass1988}}]\label{lem:bouchet3Ways}
    If $H$ is a vertex-minor of a graph~$G$ and $v$ is a vertex of~$G$ not in~$H$, then $H$ is a vertex-minor of at least one of $G-v$, $G*v-v$, and $G/v$. Moreover, the number of these three graphs which contain~$H$ as a vertex-minor is invariant under local complementation in $G$.
\end{lemma}

\subsection*{Cut-rank and rank-width}
Let $G$ be a graph, and let $A$ be the adjacency matrix of~$G$. That is, $A$ is the $V(G) \pivot V(G)$ matrix whose $(u,v)$-entry is one if $uv \in E(G)$ and zero otherwise. The \emph{cut-rank} of a set~$X\subseteq V(G)$, which is denoted by $\rho_G(X)$ (or just $\rho(X)$ if the graph is clear), is the rank over the binary field of the submatrix of $A$ with rows $X$ and columns $V(G)\setminus X$. As a function on subsets of $V(G)$, cut-rank is symmetric and submodular~\cite{RWAndVM}. Moreover, the cut-rank function is invariant under local complementation~\cite[Proposition~2.6]{RWAndVM}.

We next define rank-width, which was introduced by Oum and Seymour~\cite{approxCWBW}. A \emph{rank-decomposition} of a graph $G$ is a tree $T$ so that every vertex of~$G$ is a leaf of $T$, and every vertex of $T$ has degree at most three. The \emph{width} of an edge $e$ of $T$ is the cut-rank in $G$ of the set of all leaves of one of the components of $T-e$. The \emph{width} of $T$ is the maximum width of an edge of $T$. Finally, the \emph{rank-width} of~$G$ is the minimum, over all rank-decompositions~$T$ of~$G$, of the width of $T$. 
Note that if $H$ is a vertex-minor of~$G$, then the rank-width of $H$ is at most the rank-width of~$G$~\cite{RWAndVM}.

We require the following Grid Theorem for Vertex-Minors.

\begin{theorem}[Geelen, Kwon, McCarty, and Wollan~\cite{gridThmVM}]
\label{thm:grid}
For any circle graph $H$, there exists an integer $r(H)$ so that every graph with no vertex-minor isomorphic to $H$ has rank-width at most $r(H)$.
\end{theorem}

\noindent This theorem lets us reduce \cref{thm:main} to the case that $G$ has bounded rank-width, because the disjoint union of circle graphs is also a circle graph.

We also require the following theorem of Oum~\cite{RWandWQO}, which says that graphs of bounded rank-width are well-quasi-ordered under pivot-minors (and therefore also under vertex-minors).

\begin{theorem}[Oum~\cite{RWandWQO}]
\label{thm:wqo}
For any integer $r$ and any infinite set $\mathcal{C}$ of graphs of rank-width at most~$r$, there exist distinct graphs $G_1, G_2 \in \mathcal{C}$ so that $G_1$ is isomorphic to a pivot-minor of~$G_2$.
\end{theorem}
\noindent Oum conjectures that pivot-minors yield a well-quasi-ordering in general, that is, that we do not need the assumption that $\mathcal{C}$ has bounded rank-width in \cref{thm:wqo}; see~\cite{RWSurvey} and~\cite{McCartyThesis}. This conjecture would imply the famous theorem of Robertson and Seymour~\cite{WQOMinors} that graphs are well-quasi-ordered under minors.

\subsection*{Useful results on trees}

We complete this section by stating some lemmas on trees that we need in order to deal with the rank-decomposition tree. 

The first lemma is a type of Erd\H{o}s-P\'{o}sa property for subtrees of a fixed tree.

\begin{lemma}[Robertson and Seymour~{\cite[(8.6), (8.7)]{graphMinors5}}]
    \label{lem:tree}
    Let $T$ be a tree, let $k$ be a positive integer, and let $\mathcal{A}_1$, $\mathcal{A}_2$, $\ldots$, $\mathcal{A}_m$ be families of subtrees of $T$.
    Then at least one of the following holds. 
    \begin{enumerate}
        \item There exist subfamilies $\mathcal{B}_1 \subseteq \mathcal{A}_1$,
        $\mathcal{B}_2 \subseteq\mathcal{A}_2$, $\ldots$, $\mathcal{B}_m \subseteq\mathcal{A}_m$, each of size $k$,
        such that the subtrees in $\bigcup_{i\in[m]} \mathcal B_i$ are pairwise vertex-disjoint.
        \item There exist $i \in [m]$ and a set $X$ of less than $mk$ vertices of $T$ such that $T-X$ contains no subtree which is in $\mathcal A_i$.
    \end{enumerate}
\end{lemma}

We use the following Ramsey-type lemma on trees in order to deal with the first outcome of \cref{lem:tree}. (Given the $\mathcal{B}_1$, $\mathcal{B}_2$, $\ldots$, $\mathcal{B}_m$, we form a new tree $\widehat{T}$ from $T$ by contracting each subtree in $\mathcal{B}_1 \cup \mathcal{B}_2 \cup\ldots\cup \mathcal{B}_m$ into a single vertex. Thus we end up with disjoint sets of vertices $R_1$, $R_2$, $\ldots$, $R_m$ in $\widehat{T}$ to which we can apply the following lemma.)

\begin{lemma}\label{lem:subsets-of-tree}
    Let $T$ be a tree, let $k$ be a positive integer, and let $R_1$, $R_2$, $\ldots$, $R_m$ be pairwise disjoint sets of vertices of $T$ such that $\abs{R_i}> (mk-1)^2$ for each $i\in[m]$.
    Then there exists a subtree $T'$ of $T$ and sets $R_1'$, $R_2'$, $\ldots$, $R_m'$ such that 
    for each $i\in[m]$, $R_i'\subseteq R_i\cap V(T')$, $\abs{R_i'}=k$, 
    and each vertex in $R_i'$ has degree at most $m^2+1$ in~$T'$.
\end{lemma}
\begin{proof}
    Fix an arbitrary vertex $v$ of $T$. For each $i\in [m]$, let $\mathcal P_i$ be the set of all maximal paths in $T$ whose ends are $v$ and some vertex in~$R_i$. If $\mathcal P_i$ has at least $mk$ distinct paths, then let $T_i$ be the subtree formed by the union of the $mk$ paths in $\mathcal P_i$, and let $Q_i\subseteq R_i$ be the set of ends of the $mk$ paths in $\mathcal{P}_i$, where we choose the end that is not $v$. If $\mathcal P_i$ has less than $mk$ distinct paths, then 
    since $\abs{R_i}>(mk-1)^2$, 
    by the pigeonhole principle one of the paths, say $T_i$, contains at least $mk$ vertices of $R_i$. 
    Let $Q_i\subseteq R_i$ be a set of such $mk$ vertices on~$T_i$. 

    In either case, 
    $T_i$ has at most $mk+1$ leaves
    and every vertex in $Q_i$ has degree at most~$2$ in~$T_i$.
    If $x$ is the number of vertices of degree at least $m+2$ in $T_i$, then, by summing the degrees of the vertices and counting vertices of degree at least $m+2$, of degree between~$2$ and $m + 1$, and of degree exactly~$1$, we find that
    \[ x(m+2) + 2(\abs{V(T_i)}-x-(mk+1))
    + (mk+1) \leq 2 \abs{E(T_i)} 
    < 2 \abs{V(T_i)}.\]
    This implies that $x\leq k$, so 
    $T_i$ has at most $k$ vertices of degree at least $m+2$.
    Let $X_i$ be the set of all vertices of $T_i$ of degree at least $m+2$. 

    Let $T'$ be the subgraph of $T$ whose vertex set and edge set are the union of the vertex sets and edge sets of $T_1, T_2, \ldots, T_m$.
    Since $v\in V(T_i)$ for all $i\in[m]$, $T'$ is a subtree of~$T$.
    Let $R_i'=Q_i \setminus \bigcup_{j\in[m]\setminus\{i\}} X_j$. 
    Then $\abs{R_i'}\ge mk - (m-1)k=k$.

    Now it remains to see that for each $i \in [m]$, each vertex $w$ in $R_i'$ has degree at most $m^2+1$ in $T'$.
    Note that each $T_j$ with $j\neq i$ contains at most $m+1$ edges incident with~$w$. 
    Also, $T_i$ contains at most two edges incident with~$w$. 
    Thus, the total number of edges incident with~$w$ in~$T'$ is at most $(m+1)(m-1)+2\le m^2+1$.
\end{proof}

\section{Perturbations and robustness}
\label{sec:perturbations}
In this section, we introduce the key definitions of perturbations and robustness.

For a non-negative integer $t$, we say that a graph $G_1$ is a \emph{$t$-perturbation} of a graph $G_2$ if $V(G_1)=V(G_2)$ and there exists a graph $G$ on $\abs{V(G_1)}+t$ vertices such that both $G_1$ and $G_2$ are vertex-minors of~$G$. Note that this relation is symmetric; $G_1$ is a $t$-perturbation of~$G_2$ if and only if $G_2$ is a $t$-perturbation of $G_1$. Also note that if $G_1$ is a $t$-perturbation of~$G_2$, then any graph that is locally equivalent to $G_1$ is also a $t$-perturbation of $G_2$.

For a graph~$H$ and a non-negative integer~$t$, we say that a graph $G$ is \emph{$t$-robust for $H$} if every $t$-perturbation of~$G$ contains a vertex-minor isomorphic to~$H$. So, informally, a graph~$G$ is robust for $H$ if no small perturbation of~$G$ can destroy $H$ as a vertex-minor.

Perturbations can be ``chained together'' due to the following lemma.
\begin{lemma}
    \label{lem:pertProperties}
    Let $s,t$ be non-negative integers, and let $G_1$, $G_2$, $G_3$ be graphs. 
    If $G_2$ is an $s$-perturbation of $G_1$ and $G_3$ is a $t$-perturbation of $G_2$, then $G_3$ is an $(s+t)$-perturbation of~$G_1$.
\end{lemma}
\begin{proof}
    Observe that since $G_2$ is an $s$-perturbation of~$G_1$, there is a graph $G$ on $\abs{V(G_2)}+s$ vertices which contains $G_2$ as an induced subgraph and $G_1$ as a vertex-minor.
    Similarly, there is a graph~$G'$ on $\abs{V(G_2)}+t$ vertices that contains $G_2$ as an induced subgraph and $G_3$ as a vertex-minor. Take $G''$ to be the graph on $\abs{V(G_2)}+s+t$ vertices which is obtained by identifying $G$ and $G'$ along their common induced subgraph $G_2$. Then both $G_1$ and $G_3$ are vertex-minors of $G''$,
    and therefore $G_3$ is an $(s+t)$-perturbation of~$G_1$.
\end{proof}

An important corollary is that small perturbations cannot destroy robustness.

\begin{lemma}\label{lem:robustCommute}
    Let $H$ be a graph, and let $r$ and $t$ be non-negative integers with $t \geq r$. If $G$ is a graph that is $t$-robust for $H$, then any $r$-perturbation of~$G$ is $(t-r)$-robust for $H$.
\end{lemma}
\begin{proof}
    Let $G_1$ be an $r$-perturbation of~$G$. 
    Let $G_2$ be a $(t-r)$-perturbation of~$G_1$.
    Then by \cref{lem:pertProperties}, $G_2$ is a $t$-perturbation of~$G$. Since $G$ is $t$-robust for~$H$, we deduce that $G_2$ has a vertex-minor isomorphic to~$H$. 
    Thus, $G_1$ is $(t-r)$-robust for~$H$.
\end{proof}

We frequently make use of some specific types of perturbations throughout the paper.
Thus, the rest of this section is dedicated to proving two lemmas which help us find perturbations. 
The first lemma says that we can use a small perturbation to remove the edges crossing a set of small cut-rank.
For a graph $G$ and a subset $X$ of $V(G)$, we write $\delta_G(X)$ to denote the set of all edges having one end in $X$ and the other end in $V(G)\setminus X$.

\begin{lemma}\label{lem:pertCutRank}
    For any graph $G$ and subset $X$ of $V(G)$, the graph $G-\delta_G(X)$ is a $2\rho_G(X)$-perturbation of~$G$.
\end{lemma}
\begin{proof}
    Let $Y=V(G)\setminus X$ and $r=\rho_G(X)$. 
    We may assume that $r\neq 0$.
    By a well-known property in linear algebra, every matrix of rank $r$ is a sum of $r$ rank-one matrices. (Here and throughout the paper, we are working over the binary field.)
    Observe that a binary matrix of rank $1$ is precisely described by a nonempty set of rows and a nonempty set of columns such that the entries between those rows and columns are all $1$'s and other entries are all~$0$'s.
    Therefore, there exist nonempty $X_1, X_2, \ldots, X_r\subseteq X$ and $Y_1, Y_2, \ldots, Y_r\subseteq Y$ such that 
    a vertex $x\in X$ is adjacent to a vertex $y\in Y$ if and only if 
    there are an odd number of $i$'s such that $(x,y)\in X_i\times Y_i$.

    Let us construct a graph $\widehat G$ by adding $2r$ vertices $x_1$, $x_2$, $\ldots$, $x_r$, $y_1$, $y_2$, $\ldots$, $y_r$ to $G$ and making $x_i$ adjacent to all vertices in $X_i$ and $y_i$ adjacent to all vertices in $Y_i\cup \{x_i\}$.
    Then 
    \[ G-\delta_G(X)=\widehat G\pivot x_1y_1\pivot x_2y_2\pivot\cdots\pivot x_ry_r -\{x_1,x_2,\ldots,x_r,y_1,y_2,\ldots,y_r\}.\]  
    Therefore $G-\delta_G(X)$ is a $2r$-perturbation of~$G$.
\end{proof}

The second lemma states that for any set~$X$ of vertices of small cut-rank in a graph~$G$, all vertex-minors of $G$ with vertex set $X$ are small perturbations of each other. 
This lemma is essentially a refinement of~\cite[Theorem 1.10]{circlePivoting}. However, we provide a direct proof for the sake of readability. We suspect that this lemma holds for a linear function in the cut-rank. Still, we do not try to optimize the function since there are other roadblocks to obtaining an efficient function for \cref{thm:main} anyway (notably, our application of \cref{thm:wqo}).


\begin{lemma}\label{lem:vm-perturbation}
    Let $G$ be a graph and $X$ be a subset of $V(G)$.
    If $G_1$ and $G_2$ are vertex-minors of~$G$ which have vertex set $X$, then $G_1$ is a $(2^{\rho_G(X)+1})$-perturbation of $G_2$.
\end{lemma}
\begin{proof}
    Let $r=\rho_G(X)$ for convenience. We may assume that 
    $G$ is a vertex-minor-minimal graph containing both $G_1$ and $G_2$ as vertex-minors. 
    
    We say that a vertex $v$ of $G$ is \emph{unique for $G_1$} if $G_1$ is a vertex-minor of exactly one of the three graphs $G-v$, $G*v-v$, and $G/v$. Recall from \cref{lem:bouchet3Ways} that if $v$ is not a vertex of $G_1$, then at least one of these three graphs contains $G_1$ as a vertex-minor, and the number of these three graphs which contain $G_1$ as a vertex-minor is invariant under local complementation in $G$. We define which vertices are \emph{unique for $G_2$} analogously. 
    
    We claim that for each $i\in\{1,2\}$, at most $2^r$ vertices are unique for~$G_i$. By locally complementing in $G$, we may assume (only for this claim) that $G_i$ is an induced subgraph of $G$. Note that any binary matrix of rank $r$ has at most $2^r$ distinct columns. Suppose that $u,v\in V(G)\setminus V(G_i)$ are distinct vertices with the same neighbors in $G_i$. If $u$ and $v$ are non-adjacent, then $G_i$ is a vertex-minor of both $G-v$ and $G*v-v$ since $(G*v-v)*u$ contains $G_i$ as an induced subgraph. If $u$ and $v$ are adjacent, then $G_i$ is a vertex-minor of both $G-v$ and $G/v$, since $(G\times uv-v)$ contains $G_i$ as an induced subgraph. The claim follows.

    Now, if $G$ has more than $2^{r+1}$ vertices outside of $X$, then there exists a vertex $v$ outside of $X$ which is not unique for $G_1$ or for $G_2$. Then both $G_1$ and $G_2$ are vertex-minors of at least one of $G$, $G*v-v$, and $G/v$. This contradicts the minimality of $G$. Thus, $G$ has at most $2^{r+1}$ vertices outside of $X$. So $G$ itself shows that $G_1$ is a $(2^{r+1})$-perturbation of $G_2$, since both $G_1$ and $G_2$ are vertex-minors of $G$. This completes the proof of \cref{lem:vm-perturbation}.
\end{proof}

\section{Many small disjoint robust parts}
\label{sec:RobustParts}

This section is dedicated to proving \cref{prop:chain}. Roughly, this proposition says that in any graph of small rank-width, either we can find the desired perturbation, or we can find a vertex-minor that has many disjoint sets of vertices so that each set 
\begin{enumerate*}[label=\arabic*)]
\item induces a graph that is robust, and 
\item has bounded size.
\end{enumerate*} 
We also require the sets to have very small cut-rank; the bound on the size depends on the desired robustness, but the bound on the cut-rank does not. Throughout this section, we assume that $H$ has no isolated vertices. We deal with isolated vertices of $H$ later on.

First, we prove a lemma which gives either the desired perturbation, or many disjoint robust induced subgraphs.
Given a subtree $T'$ of a rank-decomposition~$T$ of a graph~$G$, we write $G(T')$ for the subgraph of $G$ induced by all vertices of~$G$ which are also vertices of~$T'$.

\begin{lemma}\label{lem:manyparts}
    Let $k$, $t$, and $r$ be positive integers, and let $H$ be a graph with no isolated vertices and with components $H_1$, $H_2$, $\ldots$, $H_m$.
    Then for every graph $G$ and every rank-decomposition~$T$ of~$G$ of width at most $r$, at least one of the following holds.
    \begin{enumerate}
    \item There exists $i \in [m]$ such that there is a $(4rmk+2tmk)$-perturbation of~$G$ which has no vertex-minor isomorphic to~$H_i$.
    \item There exist pairwise vertex-disjoint subtrees $(T_{i,j}:i \in [m], j \in [k])$ of~$T$ such that for all $i\in [m]$ and $j\in [k]$, the graph $G(T_{i,j})$ is $t$-robust for $H_i$.
    \end{enumerate}
\end{lemma}
\begin{proof}
    For each $i\in[m]$, let $\mathcal{A}_i$ be the set of all subtrees $T'$ of $T$ such that $G(T')$ is $t$-robust for $H_i$. Now we apply \cref{lem:tree}. First suppose that there exist subfamilies $\mathcal{B}_1 \subseteq \mathcal{A}_1,\mathcal{B}_2 \subseteq\mathcal{A}_2,\ldots,\mathcal{B}_m \subseteq\mathcal{A}_m$, each of size $k$, such that the subtrees in $\bigcup_{i\in[m]} \mathcal B_i$ are pairwise vertex-disjoint. This directly yields the second outcome of the statement. 
    
    Thus, we may assume that there exist $i \in [m]$ and a set $X$ of less than $mk$ vertices of $T$ such that $T-X$ contains no subtree in $\mathcal{A}_i$. Let $(V_1, V_2, \ldots, V_\ell)$ be the partition of $V(G)$ into nonempty sets such that 
    two vertices $u$ and $v$ are in the same part if and only if the leaves of $T$ corresponding to $u$ and $v$ are in the same component of $T-X$. (If $T$ has a leaf vertex in $X\cap V(G)$, then that vertex is in its own part.) Since $T$ is subcubic and connected, $\ell \leq 2\abs{X}+1\leq 2mk$. (Each leaf vertex of $T$ in $X\cap V(G)$ contributes one part to $(V_1, V_2, \ldots, V_{\ell})$, and deleting it from $T$ does not increase the number of components. Deleting a non-leaf vertex increases the number of components by at most two.) 
    
    Let $G'$ be the graph obtained from $G$ by removing all edges between different parts of $(V_1, V_2, \ldots, V_\ell)$. We claim that $G'$ is a $4rmk$-perturbation of $G$. To see this, note that $G'$ can be obtained from $G$ as follows. First, for each leaf vertex of $T$ in $X\cap V(G)$, we can isolate that vertex with a $2$-perturbation by~\cref{lem:pertCutRank}, and then combine the perturbations by~\cref{lem:pertProperties}. Next, consider deleting each remaining vertex in $X$ from $T$ one at a time, in an arbitrary order. Suppose that deleting the next vertex breaks a component $T'$ into three new components $T_1, T_2, T_3$. We can remove all edges between vertices in $G(T_1)$ and vertices in $G(T_2)$ or $G(T_3)$ by a $2r$-perturbation by~\cref{lem:pertCutRank}. Then we can remove all edges between vertices in $G(T_2)$ and vertices in $G(T_3)$ by a $2r$-perturbation for the same reason. Finally, we can combine the perturbations by~\cref{lem:pertProperties}. In total, the order of the perturbation is at most $2\cdot 2r\abs{X} \leq 4rmk$.
    
    Note that for each $j \in [\ell]$, we have $G'[V_j]=G[V_j]$. If $V_j$ is a singleton, then $G'[V_j]$ is not $t$-robust for $H_i$ because $H_i$ is not a singleton. Otherwise, there exists a component $T'$ of~$T-X$ so that the vertex set of $G(T')$ is $V_j$. Since $T'$ is not in $\mathcal{A}_i$, this means that $G'[V_j]$ is not $t$-robust for~$H_i$. Thus, for each $j \in [\ell]$, the graph $G'[V_j]$ has a $t$-perturbation $G_j$ that has no vertex-minor isomorphic to $H_i$.
    Let $G''$ be the disjoint union of $G_1$, $G_2$, $\ldots$, $G_\ell$.
    Then $G''$ is a $(4rmk+t\ell)$-perturbation of~$G$.
    Since $H_i$ is connected, if $G''$ has a vertex-minor isomorphic to $H_i$, then $G''$ has a component with a vertex-minor isomorphic to $H_i$.
    Since each component of $G''$ is an induced subgraph of $G_j$ for some $j\in [\ell]$, by the assumption, $G''$ has no vertex-minor isomorphic to $H_i$. 
    Since $\ell\le 2mk$, this yields the first outcome of the statement, and we are done.
\end{proof}

Now we are ready to prove the main proposition of this section. Roughly, it says that we can find the desired perturbation or many disjoint robust parts of bounded size and very small cut-rank. It is important later on that the constant $c$ in the following statement does not depend on~$k$, and that the bound on the cut-rank does not depend on $t$. 

\begin{proposition}\label{prop:chain}
    Let $k$ be a positive integer, let $t, r$ be non-negative integers, and let $H$ be a graph with no isolated vertices and with components $H_1, H_2, \ldots, H_m$. Then there exist constants $c = c_H(t,r)$ and $p=p_H(k,t,r)$ such that for every graph $G$ of rank-width at most~$r$, at least one of the following holds.
    \begin{enumerate}
    \item There exists $i \in [m]$ such that there is a $p$-perturbation of $G$ which has no vertex-minor isomorphic to $H_i$.
    \item There exists a vertex-minor $\widetilde{G}$ of $G$ whose vertex set is the union of pairwise disjoint sets $(X_{i,j}: i \in [m], j \in [k])$ such that for all $i \in [m]$ and $j \in [k]$, we have $\abs{X_{i,j}}\leq c$, the graph $\widetilde{G}[X_{i,j}]$ is $t$-robust for $H_i$, and $\rho_{\widetilde{G}}(X_{i,j})\leq r(m^2+1)$.
\end{enumerate}
\end{proposition}
\begin{proof} 
    We set $\ell = 2^{r(m^2+1)+1}$ and $p = (4rm(mk)^2+2(t+2\ell) m(mk)^2)$. So $\ell$ just depends on $H$ and $r$, and $p$ depends on $H$, $k$, $t$, and $r$. We next define the constant~$c$.
    
    For each $i \in [m]$, let $\mathcal{C}_i$ be the class of all vertex-minor minimal graphs that have rank-width at most $r$ and are isomorphic to a graph that is $(t+\ell)$-robust for $H_i$. 
    Let $\mathcal{C}_i'$ be a subfamily of $\mathcal{C}_i$ which is chosen by selecting one representative of each equivalence class under isomorphism and local equivalence.
    By \cref{thm:wqo}, if $\mathcal{C}_i'$ was infinite, then there would exist distinct $G_1, G_2 \in \mathcal{C}_i'$ so that $G_1$ is isomorphic to a vertex-minor of $G_2$. Since~$\mathcal{C}_i'$ only contains one representative of each equivalence class, $G_1$ has strictly fewer vertices than $G_2$. This contradicts the vertex-minor-minimality of~$G_2$. Thus $\mathcal{C}_i'$ is finite for each $i \in [m]$. It follows that there exists an integer $c=c_H(t,r)$ such that for each $i\in [m]$, every graph which is $(t+\ell)$-robust for $H_i$ contains a vertex-minor that has at most $c$ vertices and is still $(t+\ell)$-robust for $H_i$.

    Now, let $G$ be a graph of rank-width at most $r$, and let $T$ be a rank-decomposition of $G$ of width at most~$r$. By \cref{lem:manyparts}, either there exists $i \in [m]$ such that there is a $p$-perturbation of $G$ which has no vertex-minor isomorphic to $H_i$, or there exist pairwise vertex-disjoint subtrees $(T_{i,j}: i\in [m], j \in [(mk)^2])$ of $T$ such that for all $i \in [m]$ and $j \in [(mk)^2]$, the graph~$G(T_{i,j})$ is $(t+2\ell)$-robust for $H_i$. In the first case, the first outcome of the proposition holds. So we may assume we are in the second case.

    We now prove a key claim.

    \begin{claim}
    There is an induced subgraph $\widehat{G}$ of $G$ whose vertex set is the union of pairwise disjoint sets $(Y_{i,j}: i \in [m], j \in [k])$ such that for all $i \in [m]$ and $j \in [k]$, the graph $\widehat{G}[Y_{i,j}]$ is $(t+2\ell)$-robust for $H_i$, and $\rho_{\widehat{G}}(Y_{i,j})\leq r(m^2+1)$.
    \end{claim}
    \begin{claimproof}
    First, let $\widehat{T}$ be the tree obtained from $T$ by contracting all of the edges in any of the subtrees $(T_{i,j}:i \in [m], j \in [(mk)^2])$. For each $i \in [m]$, let $R_i\subseteq V(\widehat{T})$ be the set of $(mk)^2$-many vertices corresponding to the $(mk)^2$-many subtrees whose first index is $i$. By \cref{lem:subsets-of-tree}, there exists a subtree $\widehat{T}'$ of $\widehat{T}$ and sets $R_1'$, $R_2'$, $\ldots$, $R_m'$ such that for every~$i \in [m]$, $R_i' \subseteq R_i \cap V(\widehat{T}')$, $\abs{R_i'}=k$, and each vertex in $R_i'$ has degree at most $m^2+1$ in $\widehat{T}'$.

    For each $i \in [m]$, there are $k$ indices $j \in [(mk)^2]$ so that $T_{i,j}$ was contracted down to a vertex in $R_i'$. By reordering the second indices, we may assume that the first $k$ subtrees $T_{i,1}$, $T_{i,2}$, $\ldots$, $T_{i,k}$ correspond to vertices in $R_i'$. Then, for each $i\in [m]$ and $j \in [k]$, we let $Y_{i,j}$ denote the vertex set of $G(T_{i,j})$. We write $\widehat{G}$ for the induced subgraph of $G$ whose vertex set is the union of all of these $Y_{i,j}$. It is clear that for any $i \in [m]$ and $j \in [k]$, the graph $\widehat{G}[Y_{i,j}]$ is $(t+2\ell)$-robust for $H_i$. So we just need to consider the cut-rank of $Y_{i,j}$ in $\widehat{G}$. Since each~$Y_{i,j}$ corresponds to a vertex of $\widehat{T}'$ with degree at most $m^2+1$, there exists a set $F_{i,j} \subseteq E(T)$ of size at most~$m^2+1$ so that $Y_{i,j}$ is contained in a different component of $T-F_{i,j}$ than any other $Y_{i',j'}$. It follows from the submodularity of cut-rank that $\rho_{\widehat{G}}(Y_{i,j})\leq r(m^2+1)$, as desired.
    \end{claimproof}

    Now, by locally complementing at vertices in $Y_{1,1}$, we can obtain a graph $G_{1,1}$ from~$\widehat{G}$ which has a set $X_{1,1} \subseteq Y_{1,1}$ of size at most $c$ which induces a subgraph that is $(t+\ell)$-robust for $H_1$. By \cref{lem:vm-perturbation}, for any $i' \in [m]$ and $j' \in [k]$, we have that $G_{1,1}[Y_{i',j'}]$ is an $\ell$-perturbation of~$\widehat{G}[Y_{i',j'}]$. 
    Thus, by \cref{lem:robustCommute}, the graph $G_{1,1}[Y_{i',j'}]$ is still $(t+\ell)$-robust for $H_{i'}$. Next, by locally complementing at vertices in $Y_{1,2}$, we can obtain a graph $G_{1,2}$ with a set of vertices $X_{1,2} \subseteq Y_{1,2}$ of size at most $c$ which induces a subgraph that is $(t+\ell)$-robust for~$H_1$. Again by \cref{lem:vm-perturbation}, for any $i' \in [m]$ and $j' \in [k]$, the graph $G_{1,2}[Y_{i',j'}]$ is an $\ell$-perturbation of~$\widehat{G}[Y_{i',j'}]$, and thus still $(t+\ell)$-robust for~$H_{i'}$.
    
    We continue this process until we obtain a graph $G_{m,k}$ with a collection $(X_{i,j}:i \in [m], j \in [k])$ of pairwise disjoint sets of vertices, each of size at most $c$. 
    (The order in which we consider the sets $(Y_{i,j}: i \in [m], j \in [k])$ does not really matter. 
    However, to clarify, if $i \in [m]$ and $j \in [k-1]$, then $G_{i,j+1}$ is obtained from $G_{i,j}$ by locally complementing within~$Y_{i,j+1}$ in order to obtain a set $X_{i,j+1} \subseteq Y_{i,j+1}$ which has size at most $c$ and induces a graph that is $(t+\ell)$-robust for $H_i$. 
    Similarly, if $i \in [m-1]$, then $G_{i+1, 1}$ is obtained from~$G_{i, k}$ by locally complementing within $Y_{i+1,1}$ in order to obtain a set $X_{i+1,1} \subseteq Y_{i+1,1}$ which has size at most~$c$ and induces a graph that is $(t+\ell)$-robust for $H_{i+1}$.)
    
    Let $\widetilde{G}$ denote the subgraph of $G_{m,k}$ induced on the union of all $X_{i,j}$. Note that for all $i \in [m]$ and $j \in [k]$, we have $\rho_{\widetilde{G}}(X_{i,j}) \leq r(m^2+1)$ since $G_{m,k}$ is locally equivalent to $\widehat{G}$ and each~$X_{i,j}$ is a subset of~$Y_{i,j}$.
    It remains to show that for all $i \in [m]$ and $j \in [k]$, the graph $\widetilde{G}[X_{i,j}]=G_{m,k}[X_{i,j}]$ is $t$-robust for $H_i$. 
    We know from construction that $G_{i,j}[X_{i,j}]$ is $(t+\ell)$\nobreakdash-robust for~$H_i$. 
    Consider the graph $G'$ obtained from $G_{i,j}$ by deleting the vertices in~$Y_{i,j}\setminus X_{i,j}$. Notice that $\rho_{G'}(X_{i,j})\leq \rho_{\widehat{G}}(Y_{i,j})\leq r(m^2+1)$. 
    Also notice that, since we never again locally complement at vertices in~$Y_{i,j}$, both $G_{i,j}[X_{i,j}]$ and $G_{m,k}[X_{i,j}]$ are vertex-minors of~$G'$. 
    So by \cref{lem:vm-perturbation}, the graph~$G_{m,k}[X_{i,j}]$ is an $\ell$-perturbation of~$G_{i,j}[X_{i,j}]$. Thus, by \cref{lem:robustCommute}, the graph~$G_{m,k}[X_{i,j}]$ is $t$-robust for $H_i$, as desired. This proves \cref{prop:chain}.
\end{proof}

\section{Extracting \texorpdfstring{$kH$}{kH} as a vertex-minor}
\label{sec:extract}

This section is dedicated to proving \cref{prop:cleaningChains}, which shows how to extract $kH$ from the outcome of~\cref{prop:chain}. 

We introduce some new terminology on ordered sets of vertices, helping us to apply Ramsey-type lemmas.
An \emph{ordered set} is a finite sequence $X=(x_1, x_2, \ldots, x_c)$ of distinct elements. 
For each $i\in [c]$, we write $X(i)$ for the $i$th element~$x_i$ in~$X$. 
Let $G$ be a graph.
A \emph{chain} is a sequence $\mathcal{X}=(X_1, X_2, \ldots, X_k)$ of pairwise disjoint ordered subsets of~$V(G)$. 
The \emph{length} of $\mathcal{X}$ is its number of sets $k$, and the \emph{width} of $\mathcal{X}$ is $\max_{i\in [k]}\abs{X_i}$. If $\mathcal{Y}$ is a subsequence of $\mathcal{X}$, then we say that $\mathcal{Y}$ is a \emph{subchain} of~$\mathcal{X}$. We write $V(\mathcal{X})$ for $\bigcup_{i\in[k]} X_i$. We say that a chain is \emph{$c$-uniform} if each $X_i$ has size exactly~$c$. 

Let $\mathcal{X}=(X_1, X_2, \ldots, X_k)$ be a $c$-uniform chain.
For each $j\in [c]$, we write $\mathcal{X}(j)$ for the set of all $j$th elements, that is, $\mathcal{X}(j)=\{X_1(j), X_2(j), \ldots, X_k(j)\}$.
Let $j_1, j_2\in [c]$ be integers (possibly with $j_1=j_2$). 
We say that the pair $(j_1, j_2)$ is \emph{fixed} with respect to~$\mathcal{X}$ and~$G$ if for any distinct $i_1, i_2 \in [k]$, there is no edge between $X_{i_1}(j_1)$ and $X_{i_2}(j_2)$ in $G$. That is, every edge between $\mathcal{X}(j_1)$ and $\mathcal{X}(j_2)$ has both ends in one of the parts $X_1$, $X_2$, $\ldots$, $X_k$. In this section, we show how to locally complement and take subchains in order to fix every pair $(j_1, j_2)$ with $j_1, j_2 \in [c]$.

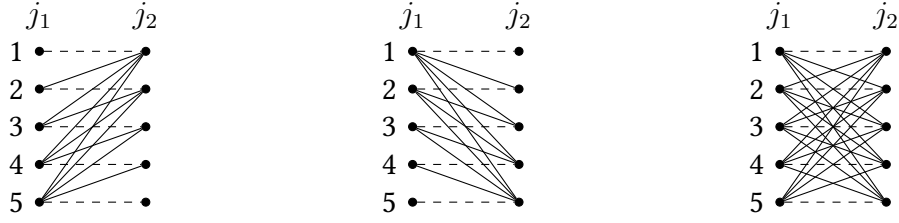
\begin{figure}
    \centering
    \begin{subfigure}{0.31\textwidth}
        \centering
        \begin{tikzpicture}
        \tikzstyle{v}=[circle, draw, solid, fill=black, inner sep=0pt, minimum width=3pt]
        \node [label=$j_1$] (v) at (-.7,2.5) {};
        \node [label=$j_2$] (w) at (.7,2.5){};
        \foreach \i in {1,2,3,4,5}{
            \node [v, label=left:\i] (v\i) at (-.7,3-\i*.5) {};
            \node [v] (w\i) at (.7,3-\i*.5) {};
            \draw[dashed] (v\i)--(w\i);
            \foreach \j in {1,2,3,4,5}{
                \ifthenelse{\j<\i}{\draw(v\i)--(w\j);}{}
            }
        }
    \end{tikzpicture}
    \caption{An up-coupled half graph}
    \end{subfigure}
   \begin{subfigure}{0.31\textwidth}
        \centering
        \begin{tikzpicture}
        \tikzstyle{v}=[circle, draw, solid, fill=black, inner sep=0pt, minimum width=3pt]
        \node [label=$j_1$] (v) at (-.7,2.5) {};
        \node [label=$j_2$] (w) at (.7,2.5){};
        \foreach \i in {1,2,3,4,5}{
            \node [v, label=left:\i] (v\i) at (-.7,3-\i*.5) {};
            \node [v] (w\i) at (.7,3-\i*.5) {};
            \draw[dashed] (v\i)--(w\i);
            \foreach \j in {1,2,3,4,5}{
                \ifthenelse{\j>\i}{\draw(v\i)--(w\j);}{}
            }
        }
    \end{tikzpicture}
    \caption{A down-coupled half graph}
    \end{subfigure}
     \begin{subfigure}{0.3\textwidth}
        \centering
        \begin{tikzpicture}
        \tikzstyle{v}=[circle, draw, solid, fill=black, inner sep=0pt, minimum width=3pt]
        \node [label=$j_1$] (v) at (-.7,2.5) {};
        \node [label=$j_2$] (w) at (.7,2.5){};
        \foreach \i in {1,2,3,4,5}{
            \node [v, label=left:\i] (v\i) at (-.7,3-\i*.5) {};
            \node [v] (w\i) at (.7,3-\i*.5) {};
            \draw[dashed] (v\i)--(w\i);
            \foreach \j in {1,2,3,4,5}{
                \ifthenelse{\j>\i}{\draw(v\i)--(w\j);}{}
            }
            \foreach \j in {1,2,3,4,5}{
                \ifthenelse{\j<\i}{\draw(v\i)--(w\j);}{}
            }
        }
    \end{tikzpicture}
    \caption{A complete couple}
    \end{subfigure}
    \caption{Coupled pairs $(j_1, j_2)$ with $j_1\neq j_2$. Two vertices in the same set $X_i\in \mathcal{X}$ may or may not be adjacent, and this is presented by a dashed line.}
    \label{fig:couple}
\end{figure}

We need to be more precise about the potential outcomes of applying Ramsey's Theorem. So, let $\mathcal{X}=(X_1, X_2, \ldots, X_k)$ be a $c$-uniform chain. Let $j_1, j_2\in [c]$ be distinct integers.
We say that $(j_1, j_2)$ is 
\begin{itemize}
    \item an \emph{up-coupled half graph} if for all distinct $i_1, i_2\in [k]$ with $i_1<i_2$, $X_{i_1}(j_1)$ is not adjacent to $X_{i_2}(j_2)$ and $X_{i_1}(j_2)$ is adjacent to $X_{i_2}(j_1)$,
    \item a \emph{down-coupled half graph} if for all distinct $i_1, i_2\in [k]$ with $i_1<i_2$, $X_{i_1}(j_1)$ is adjacent to $X_{i_2}(j_2)$ and $X_{i_1}(j_2)$ is not adjacent to $X_{i_2}(j_1)$, and
    \item a \emph{complete couple} if 
    for all distinct $i_1, i_2\in [k]$ with $i_1<i_2$, $X_{i_1}(j_1)$ is adjacent to $X_{i_2}(j_2)$ and $X_{i_1}(j_2)$ is adjacent to $X_{i_2}(j_1)$.
\end{itemize}
 See~\cref{fig:couple} for illustrations. If $(j_1, j_2)$ is fixed or one of the three types above, then we call $(j_1, j_2)$ \emph{coupled}.

Let $c$ be an integer. We fix pairs $(j_1, j_2)$ with $j_1, j_2\in [c]$ by considering them in lexicographic order. To be more precise, the \emph{lexicographic order} is the total order on $\{(j_1,j_2):1\leq j_1\leq j_2\leq c\}$ where pairs are ordered first by their first index and second by their second index. So all pairs $(1,j_2)$ come first, followed by the remaining pairs of the form $(2, j_2)$, and so on. Note that a pair $(j_1, j_2)$ is fixed with respect to some chain $\mathcal{X}$ and graph $G$ if and only if $(j_2, j_1)$ is fixed with respect to $\mathcal{X}$ and $G$. 

The next lemma shows how to fix pairs one at a time in lexicographic order. Given integers $s$ and $k$, we write $R_k(s)$ for the $k$-color Ramsey number of $s$; so every complete graph on at least $R_k(s)$ vertices whose edges are colored with at most $k$ colors contains a monochromatic complete subgraph on $s$ vertices.



    

\newcommand{\inside}{3k}
\begin{lemma}\label{lem:cleanupchain}
    Let $c$, $k$, and $\ell$ be integers with $0\le \ell < \binom{c}{2}+c$. Let $G$ be a graph, and let $\mathcal{X}=(X_1, X_2, \ldots, X_{R_4(\inside)})$ be a $c$-uniform chain of length $R_4(\inside)$ in $G$ such that the first~$\ell$ pairs in $\{(j_1,j_2):1\leq j_1\leq j_2\leq c\}$ with respect to the lexicographic order are fixed. Then 
    there exist a graph $\widetilde{G}$ and a subchain $\mathcal{Y}$ of $\mathcal{X}$ of length $k$ such that
    \begin{itemize}
        \item $\widetilde{G}$ is obtained from $G$ by applying local complementations at vertices in $V(\mathcal{X})\setminus V(\mathcal{Y})$, 
        \item the first $\ell+1$ pairs in the lexicographic order are fixed.
    \end{itemize}
\end{lemma}
\begin{proof}
Let $(j_1,j_2)$ with $1\leq j_1\leq j_2 \leq c$ be the smallest pair (in lexicographic order) that is not fixed with respect to $\mathcal{X}$ and~$G$. 

First, we take care of the case that $j_1=j_2$. By Ramsey's Theorem and since $R_4(\inside)\geq R_2(k+1)$ when $k \geq 1$, there exists a subchain $\mathcal{Z}=(Z_1, Z_2, \ldots, Z_{k+1})$ of $\mathcal{X}$ of length~$k+1$ such that $\mathcal{Z}(j_1)$ is either a clique or an independent set.
We may assume it is a clique, since if it is an independent set we are already done.
Now, let $\widetilde{G}=G*Z_1(j_1)$ and $\mathcal{Y}=(Z_2,Z_3,\ldots,Z_{k+1})$.
From the definition of local complementation, we have that $(j_1, j_1)$ is fixed with respect to~$\mathcal{Y}$. 
So we just need to consider each pair $(j'_1, j'_2)$ with $1 \leq j'_1 \leq j'_2 \leq c$ which comes before $(j_1, j_1)$ in lexicographic order.
It follows that $j'_1 <j_1$, and that $(j'_1, j_1)$ is fixed with respect to $\mathcal{X}$.
Thus the vertex $Z_1(j_1)$ has no edges to $\mathcal{Y}(j'_1)$, and locally complementing at $Z_1(j_1)$ does not change the neighborhood of any vertex in $\mathcal{Y}(j'_1)$.
So $(j'_1, j'_2)$ is fixed in $\mathcal{Y}$, which completes this case. 

Finally, we take care of the case that $j_1 < j_2$.
By Ramsey's Theorem, there exists a subchain $\mathcal{Z}=(Z_1, Z_2, \ldots, Z_{\inside})$ of $\mathcal{X}$ of length $\inside$ such that $(j_1, j_2)$ is coupled with respect to~$\mathcal{Z}$.
From here on out, we only use one operation: pivot on an edge joining a $j_1$th vertex and a $j_2$th vertex from different parts of~$\mathcal{Z}$, and remove those two parts from the chain. 

First, we show that this operation is ``safe'' with respect to previously fixed pairs.
That is, we prove that for any subchain $\mathcal{Y}$ and graph $\widetilde{G}$ which are obtained from $\mathcal{Z}$ and $G$ (respectively) by performing a sequence of these operations, the first $\ell$ pairs $(j'_1, j'_2)$ in lexicographic order are still fixed in $\mathcal{Y}$ and $\widetilde{G}$.
So, consider a graph $\widetilde{G}$ which is obtained from $G$ by pivoting on an edge between a $j_1$th vertex $u$ and a $j_2$th vertex $v$ in different parts of $\mathcal{Z}$. Let $\mathcal{Y}$ denote the subchain of $\mathcal{Z}$ obtained by removing those two parts.

Let $(j'_1, j'_2)$ with $1 \leq j'_1 \leq j'_2 \leq c$ be a pair which comes strictly before $(j_1, j_2)$ in lexicographic order.
First, suppose that $j'_1<j_1$. 
Then both $(j'_1, j_1)$ and $(j'_1, j_2)$ are fixed with respect to~$\mathcal{Z}$ and~$G$. 
Thus $G$ has no edges joining $u$ or $v$ to any vertex in~$\mathcal{Y}(j'_1)$. 
It follows that every vertex in~$\mathcal{Y}(j'_1)$ has the same neighborhood in $\widetilde{G}$ as in~$G$. 
Thus $(j'_1, j'_2)$ is still fixed in $\mathcal{Y}$ and $\widetilde{G}$, as desired.
For the final case, suppose that $j'_1=j_1$.
Then, since $(j'_1, j'_2)$ comes before $(j_1, j_2)$ in lexicographic order, both $(j_1, j'_1)$ and $(j_1, j'_2)$ are fixed in $\mathcal{Z}$ and $G$.
Thus $G$ has no edge joining $u$ to any vertex in $\mathcal{Y}(j'_1)\cup \mathcal{Y}(j'_2)$. 
It follows that $\widetilde{G}$ and $G$ have the same induced subgraph on $\mathcal{Y}(j'_1)\cup \mathcal{Y}(j'_2)$.
So $(j'_1, j'_2)$ is fixed with respect to~$\mathcal{Y}$ and~$\widetilde{G}$, as we claimed.

To complete the proof of~\cref{lem:cleanupchain}, it now suffices to show that we can use this ``safe'' operation in order to fix $(j_1, j_2)$.
If $(j_1, j_2)$ is fixed with respect to~$\mathcal{Z}$ and~$G$, then we are already done. 
If $(j_1, j_2)$ is a complete couple in $\mathcal{Z}$ and $G$, then we only need to perform this operation once.
That is, we pivot on an edge between the $j_1$th vertex~$u$ and the $j_2$th vertex~$v$ in different parts of~$\mathcal{Z}$ and remove those two parts from $\mathcal{Z}$.
Let $\widetilde{G}$ denote the resulting graph and $\mathcal{Y}$ the resulting chain. 
Since $(j_1, j_1)$ is fixed in $G$, the set $\mathcal{Y}(j_1)$ is contained in~$N_G(v) \setminus N_G(u)$. 
Moreover, the set $\mathcal{Y}(j_2)$ is contained in~$N_G(u)$.
It follows that $(j_1, j_2)$ is fixed in $\mathcal{Y}$ and $\widetilde{G}$, as desired.

For the final two cases, suppose that $(j_1, j_2)$ is either an up-coupled half graph or a down-coupled half graph.
Since we are only interested in how pivoting on the ``half graph edges'' affects the pair $(j_1, j_2)$, these cases are symmetric.
(Imagine flipping the roles of $j_1$ and $j_2$.)
We write the proof for the case that $(j_1, j_2)$ is an up-coupled half graph in $\mathcal{Z}$ and $G$.
We let $\widetilde{G}$ denote the graph obtained from $G$ by first pivoting on the edge between $Z_1(j_2)$ and $Z_3(j_1)$, and then pivoting on the edge between $Z_4(j_2)$ and $Z_6(j_1)$, and then pivoting on the edge between $Z_7(j_2)$ and $Z_9(j_1)$ and so on.
Then the corresponding chain $\mathcal{Y}$ equals $(Z_2, Z_5, Z_8, \ldots, Z_{3k-1})$. 

\begin{figure}
    \centering
    \begin{subfigure}{0.31\textwidth}
        \centering
        \begin{tikzpicture}
        \pgfdeclarelayer{background}
        \pgfdeclarelayer{foreground}
        \pgfsetlayers{background,main,foreground}
        \tikzstyle{v}=[circle, draw, solid, fill=black, inner sep=0pt, minimum width=3pt]
        \tikzset{c1/.style={purple, line width=6pt,opacity=0.5,line cap=round,shorten >=-2pt, shorten <=-2pt}}
        
        \node [label=$j_1$] (v) at (-.7,4.5) {};
        \node [label=$j_2$] (w) at (.7,4.5){};
        \foreach \i in {2,4,5,6,7,8,9}{
            \node [v, label=left:\i] (v\i) at (-.7,5-\i*.5) {};
            \node [v] (w\i) at (.7,5-\i*.5) {};
            \draw[dashed] (v\i)--(w\i);
            \foreach \j in {4,5,6,7,8,9}{
                \ifthenelse{\j<\i}{\draw(v\i)--(w\j);}{}
            }
        }
        \begin{pgfonlayer}{background}
            \draw[c1] (v6)--(w4);
        \end{pgfonlayer}
    \end{tikzpicture}
    \caption{$\widetilde{G}_1$}
    \end{subfigure}
    \qquad
     \begin{subfigure}{0.31\textwidth}
        \centering
        \begin{tikzpicture}
        \tikzstyle{v}=[circle, draw, solid, fill=black, inner sep=0pt, minimum width=3pt]
        \tikzset{c1/.style={purple, line width=5pt,opacity=0.8,line cap=round}}
        
        \node [label=$j_1$] (v) at (-.7,4.5) {};
        \node [label=$j_2$] (w) at (.7,4.5){};
        \foreach \i in {2,5,7,8,9}{
            \node [v, label=left:\i] (v\i) at (-.7,5-\i*.5) {};
            \node [v] (w\i) at (.7,5-\i*.5) {};
            \draw[dashed] (v\i)--(w\i);
            \foreach \j in {7,8,9}{
                \ifthenelse{\j<\i}{\draw(v\i)--(w\j);}{}
            }
        }
    \end{tikzpicture}
    \caption{$\widetilde{G}_2$}
    \end{subfigure}
    \caption{The procedure of fixing a pair $(j_1, j_2)$ with $j_1<j_2$ that is an up-coupled half graph in~\cref{lem:cleanupchain}. The graph $\widetilde{G}_2$ is obtained from $\widetilde{G}_1$ by pivoting the edge joining~$Z_4(j_2)$ and~$Z_6(j_1)$, which is depicted as an edge marked in red. Note that  $\mathcal{Y}_1(j_1)$ is independent in $\widetilde{G}_1$, and thus, when pivoting, there is no change on $(Z_7(j_1)\cup Z_7(j_2))\cup \cdots \cup (Z_{3k}(j_1)\cup Z_{3k}(j_2))$, and the edges between $\{Z_5(j_2)\}$ and $\{Z_7(j_1), \ldots, Z_{3k}(j_1)\}$ are removed. }
    \label{fig:cleanchain}
\end{figure}
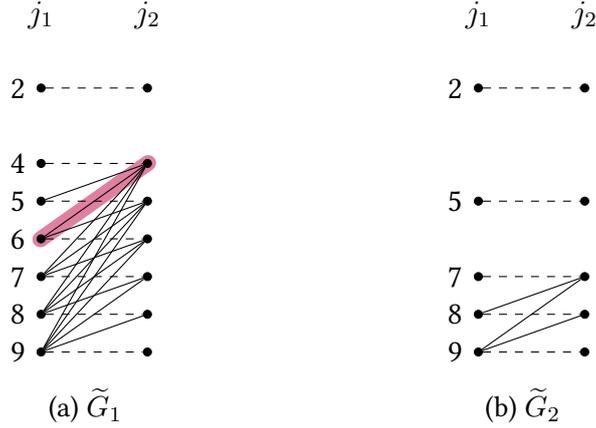

Let $\widetilde{G}_i$ denote the graph obtained by performing the first $i$ pivots described above (so $\widetilde{G}_0 = G$).
See~\cref{fig:cleanchain} for an illustration. 
Also, let $\mathcal{Y}_i$ denote the chain obtained from $\mathcal{Z}$ by removing the $2i$ parts involved in the first $i$ pivots (so $\mathcal{Y}_0 = \mathcal{Z}$). 
We claim that $(j_1, j_2)$ is an up-coupled half graph with respect to the chain $(Z_{3i+1}, Z_{3i+2}, Z_{3i+3}, \ldots, Z_{\inside})$ and graph~$\widetilde{G}_i$.
We also claim that $\widetilde{G}_i$ has no edge joining the $j_1$th or $j_2$th vertex of any of the first $i$ parts of $\mathcal{Y}_i$ (which are $Z_2, Z_5, \ldots, Z_{3i-1}$) to any $j_1$th or $j_2$th vertex in a different part of $\mathcal{Y}_i$. 
We prove these claims by induction on~$i$. 
The base case of $i=0$ trivially holds, and the final case of $i=k$ implies that $(j_1, j_2)$ is fixed with respect to $\mathcal{Y}$ and $\widetilde{G}$.

So, suppose that the claim holds for $i$.
Then the graph $\widetilde{G}_{i+1}$ is obtained from $\widetilde{G}_i$ by pivoting on the edge between $u \coloneqq Z_{3i+1}(j_2)$ and $v \coloneqq Z_{3i+3}(j_1)$.
Since neither $u$ nor $v$ has any edges to the $j_1$th or $j_2$th vertices of $Z_2, Z_5, \ldots, Z_{3i-1}$ by assumption, we only need to worry about what happens to the later parts $Z_{3i+1}, Z_{3i+2}, \ldots, Z_{\inside}$.
Note that, except for possibly the parts containing $u$ and $v$ which we are not concerned about, the $j_1$th vertices of $Z_{3i+1}, Z_{3i+2}, \ldots, Z_{\inside}$ are in $N_{\widetilde{G}_i}(u)\setminus N_{\widetilde{G}_i}(v)$.
Therefore, adjacency between one of these $j_1$th vertices and another vertex $x\notin \{u,v\}$ is complemented if and only if $x \in N_{\widetilde{G}_i}(v)$.
Since $(j_1, j_2)$ is an up-coupled half-graph with respect to $(Z_{3i+1}, Z_{3i+2}, Z_{3i+3}, \ldots, Z_{\inside})$, the only $j_2$th vertex (among the sets we care about) which is in $N_{\widetilde{G}_i}(v)$ is the vertex $Y_{3i+2}(j_2)$.

The claim follows by induction on $i$.
This also completes the proof of \cref{lem:cleanupchain}.
\end{proof}

Now we obtain the following proposition as a corollary of \cref{lem:cleanupchain}.

\begin{proposition}\label{prop:cleaningChains}
    For any positive integers $c$ and $k$, and any graph $H$ with components $H_1$, $H_2$, $\ldots$, $H_m$, there exists an integer $K = K_H(c,k)$ so that the following holds. Let $G$ be a graph whose vertex set is the union of pairwise disjoint sets $(X_{i,j}: i\in [m], j \in [K])$ such that for all $i \in [m]$ and $j \in [k]$, the set $X_{i,j}$ has size at most $c$, the graph $G[X_{i,j}]$ is $(2^{r(m^3+m)+1})$-robust for $H_i$, and $\rho_G(X_{i,j})\leq r(m^2+1)$. Then $G$ contains $kH$ as a vertex-minor.
\end{proposition}
\begin{proof}
    For each $j\in [K]$, there are at most $cm$ possibilities for $\sum_{i \in [m]} \abs{X_{i,j}}$.
    So by the pigeonhole principle, and thinking of $K$ as any sufficiently large integer, we may assume that this sum of sizes is the same for every $j \in [K]$. Now, for each $j\in [K]$, we set $X_j = \bigcup_{i \in [m]}X_{i,j}$, and we arbitrarily order $X_j$. Let $\mathcal{X}=(X_1, X_2, \ldots, X_K)$. Then $\mathcal{X}$ is a chain of length~$K$. 
    
    Note that $\mathcal{X}$ is $C$-uniform for some positive integer $C \leq cm$. Also note that for every $i\in [m]$ and $j\in [K]$, the graph $G[X_{j}]$ is $(2^{r(m^3+m)+1})$-robust for $H_i$ because $X_{i,j}\subseteq X_j$. 
    Moreover, for every $j\in [K]$, we have $\rho_{G}(X_{j})\le r(m^3+m)$ by the submodularity of the cut-rank function. By applying \cref{lem:cleanupchain} at most $\binom{C}{2}+C$ times to~$\mathcal X$, we obtain a graph~$\widetilde{G}$ that is locally equivalent to~$G$ 
    and a subchain $\mathcal{Y}=(Y_1,Y_2,\ldots,Y_{km})$ of $\mathcal{X}$ of length $km$ 
    such that every pair $(j_1, j_2)$ with $1\le j_1\le j_2\le C$ is fixed with respect to $\mathcal{Y}$ and $\widetilde{G}$.
    This means that $Y_j$ has no edges to $Y_{j'}$ for any distinct $j,j' \in [km]$.
    For each $j\in [km]$, the graph $\widetilde{G}[Y_j]$ is a $(2^{r(m^3+m)+1})$-perturbation of $G[Y_j]$ by \cref{lem:vm-perturbation}.
    Thus, by \cref{lem:robustCommute}, the graph~$\widetilde{G}[Y_j]$ is $0$-robust for $H_i$ for each $i \in [m]$.
    In particular, $\widetilde{G}[Y_j]$ has each of $H_1$, $H_2$, $\ldots$, $H_m$ as a vertex-minor.
    It follows that $\widetilde{G}$, and therefore $G$, contains $kH$ as a vertex-minor.
\end{proof}

\section{Proving the main theorem}
\label{sec:conclusion}
In this section, we prove \cref{thm:main} and its rough converse, \cref{prop:converse}. 
First, we prove our main theorem, which is restated below for convenience.

\maintheorem*
\begin{proof}
    First, suppose that $H$ has $m$ isolated vertices for some $m>0$. Let $H'$ be the graph obtained from $H$ by removing all isolated vertices. We may take $t(k,H)\coloneqq t(k+km,H')$.
    
    Therefore, we may assume that $H$ has no isolated vertices. Let $G$ be a graph; we may assume that $G$ has no vertex-minor isomorphic to $kH$. By~\cref{thm:grid}, there exists $r\coloneqq r(kH)$ such that every graph with no vertex-minor isomorphic to $kH$ has rank-width at most $r$. (Note that $kH$ is a circle graph since circle graphs are closed under taking disjoint unions.) Thus, the rank-width of~$G$ is at most~$r$.

    Let $H_1$, $H_2$, $\ldots$, $H_m$ be the components of $H$. For convenience, set $\ell =2^{r(m^3+m)+1}$. Let $c_H$ and $p_H$ be the functions given by \cref{prop:chain}, and let $K_H$ be the function given by \cref{prop:cleaningChains}. Set $c \coloneqq c_H(\ell,r)$, $K\coloneqq K_H(c,k)$, and $p \coloneqq p_H(K,\ell, r)$. By \cref{prop:chain}, at least one of the following holds. 
     \begin{enumerate}
    \item There exists $i \in [m]$ such that there is a $p$-perturbation of~$G$ which has no vertex-minor isomorphic to $H_i$.
    \item There exist 
    a vertex-minor $\widetilde{G}$ of~$G$ whose vertex set is the union of pairwise disjoint sets $(X_{i,j}: i\in [m], j \in [K])$ such that for all $i\in [m]$ and $j\in [K]$, we have $\abs{X_{i,j}}\le c$, the graph $\widetilde{G}[X_{i,j}]$ is $\ell$-robust for $H_i$, 
    and $\rho_{\widetilde{G}}(X_{i,j})\le r(m^2+1)$. 
    \end{enumerate}
    In the first case, we have found the desired perturbation and are done. In the second case, $G$ contains $kH$ as a vertex-minor by \cref{prop:cleaningChains}, and again we are done.
\end{proof}

Now we prove the converse, \cref{prop:converse}, which we restate below for convenience.

\mainProp*
\begin{proof}
        We proceed by induction on $t$. The base case holds since when $t=0$, the graphs~$G$ and~$\widetilde{G}$ are locally equivalent. Now suppose that it holds for~$t-1$.
    
    Suppose towards a contradiction that $G$ has a vertex-minor isomorphic to $(t+1)H$. Let~$\widehat{G}$ be a graph which contains both $G$ and $\widetilde{G}$ as vertex-minors and has $\abs{V(G)}+t$ vertices. By locally complementing in $\widehat{G}$, we may assume that $(t+1)H$ is an induced subgraph of $\widehat{G}$ on a set of vertices contained in $V(G)$. By \cref{lem:bouchet3Ways}, for each vertex $v$ of $\widehat{G}$ not in $\widetilde{G}$, the graph $\widetilde{G}$ is a vertex-minor of at least one of the three graphs $\widehat{G}-v$, $\widehat{G}*v-v$, and~$\widehat{G}/v$.

    Thus $\widetilde{G}$ is a $(t-1)$-perturbation of at least one of the three graphs $(\widehat{G}-v)[V(G)]$, $(\widehat{G}*v-v)[V(G)]$, and~$(\widehat{G}/v)[V(G)]$. So, to obtain a contradiction, it suffices to show that each of these three graphs contains $tH$ as a vertex-minor. Certainly $(t+1)H$ is still an induced subgraph of $\widehat{G}-v$. Moreover, if $v$ has no neighbors in $\widehat{G}$ which are in the copy of~$(t+1)H$, then $(t+1)H$ is still an induced subgraph of both $\widehat{G}*v-v$ and $\widehat{G}/v$.
    
    So we may assume that in $\widehat{G}$, the vertex $v$ is adjacent to a vertex $u$ which is in the copy of $(t+1)H$. By complementing at $u$ and removing the copy of $H$ which contains $u$, we can see that $(\widehat{G}*v-v)[V(G)]$ contains $tH$ as a vertex-minor. By pivoting on $uv$ and removing the copy of $H$ that contains $u$, we can see that $(\widehat{G}/v)[V(G)]$ contains $tH$ as a vertex-minor as well. This completes the proof.
\end{proof}

\section{Discussion on perturbations}
\label{sec:discussion}

In this section, we discuss alternate ways of defining perturbations.
Additionally, we prove that we do not get an Erd\H{o}s-P\'{o}sa-property for vertex deletions and local complementations (\cref{lem:pertNecessary}), thus perturbations are indeed necessary.

Given a graph $G$, let us write $A(G)$ to denote the adjacency matrix of $G$ over the binary field.
A graph $G_1$ is a \emph{rank-$t$ perturbation} of a graph~$G_2$ if the adjacency matrix of~$G_1$ can be obtained from the adjacency matrix of $G_2$ by adding (over the binary field) a matrix of rank at most $t$ and then changing all diagonal entries to~$0$. 

\begin{lemma}\label{lem:rank}
    Every rank-$t$ perturbation of a graph~$G$ is a $t$-perturbation of $G$.
\end{lemma}
\begin{proof}
    Let $G'$ be a rank-$t$ perturbation of~$G$. Then there exist a matrix $M$ over the binary field and a diagonal matrix $D$ over the binary field such that $A(G') = A(G) + M + D$, where the rank of $M$ is at most $t$.
    Since adjacency matrices are symmetric, $M$ is symmetric.
    It is well known that a symmetric matrix of rank $r$ over any field can be written as a sum of~$r-2q$ rank-$1$ symmetric matrices and $q$ rank-$2$ symmetric matrices for some integer $q$; see Godsil and Royle~\cite[Lemma 8.9.3]{GR2001}. Now we follow an analysis of rank-$1$ and rank-$2$ matrices over the binary field in the proof of~\cite[Theorem 1.1]{NguyenOum20}.
    \begin{itemize}
        \item Every symmetric matrix of rank $1$ over the binary field is of the form 
        \[ 
        \begin{pmatrix}
            1&0\\
            0&0
        \end{pmatrix}, 
        \]
        where $0$ represents an all-$0$ matrix and $1$ represents an all-$1$ matrix. 
        \item Every symmetric matrix of rank $2$ over the binary field is of the form 
        \[ 
        \begin{pmatrix}
            0&1&0\\
            1&0&0\\ 
            0&0&0
        \end{pmatrix} \text{ or } 
        \begin{pmatrix}
        0&1&1&0\\
        1&0&1&0\\
        1&1&0&0\\
        0&0&0&0
        \end{pmatrix}.
        \]
    \end{itemize}

    Let $r$ be the rank of $M$. Then for some integer $q$, there exist $r-2q$ rank-$1$ symmetric matrices $M_1$, $M_2$, $\ldots$, $M_{r-2q}$ and $q$ rank-$2$ symmetric matrices $N_1$, $N_2$, $\ldots$, $N_q$ such that $M=M_1+M_2+\cdots+M_{r-2q}+N_1+N_2+\cdots+N_q$. For each $i\in [r-2q]$, there exists a maximal $X_i\subseteq V(G)$ such that the $X_i\times X_i$ submatrix of $M_i$ is an all-$1$ matrix.
    Likewise, for each $j\in [q]$, there exist $Y_j,Z_j\subseteq V(G)$ such that 
    \[ 
    N_j=
    \bordermatrix{
        & Y_j\setminus Z_j & Y_j\cap Z_j & Z_j\setminus Y_j & V(G)\setminus (Y_j\cup Z_j) \cr
        Y_j\setminus Z_j& 0& 1 & 1&0\cr
        Y_j\cap Z_j & 1&0&1&0\cr
        Z_j\setminus Y_j & 1 &1&0&0\cr
        V(G)\setminus(Y_j\cup Z_j) &0&0&0&0
    }.
    \]

    Now we explain how to obtain $G'$ from $G$ by an  $r$-perturbation.
    To do so, we obtain a new graph $\widehat{G}$ from $G$ by adding $r$ new vertices as follows. 
    For each $i\in [r-2q]$, we add a vertex $x_i$ adjacent to every vertex in $X_i$.
    For each $j\in [q]$, we add two vertices $y_j$, $z_j$ so that~$y_j$ is adjacent to every vertex in $Y_j\cup \{z_j\}$ 
    and 
    $z_j$ is adjacent to every vertex in $Z_j\cup \{y_j\}$. 
    Now $G'$ is obtained from $\widehat{G}$ by applying local complementations at $x_1$, $x_2$, $\ldots$, $x_{r-2q}$ and applying pivots on $y_1z_1$, $y_2z_2$, $\ldots$, $y_qz_q$.
\end{proof}

The converse of \cref{lem:rank} is false in general, but the following rough converse holds.

\begin{lemma}\label{lem:rankConv}
    Any $t$-perturbation of a graph~$G$ is locally equivalent to a rank-$2t$ perturbation of~$G$.
\end{lemma}
\begin{proof}
    We proceed by induction on~$t$. It is trivial if $t=0$. So we may assume that $t>0$. Let $G'$ be a $t$-perturbation of $G$.
    Then there is a graph~$\widehat{G}$ on $\abs{V(G)}+t$ vertices such that both $G'$ and $G$ are vertex-minors of~$\widehat{G}$.
    By applying local complementations in $\widehat{G}$, 
    we may assume that $G$ is an induced subgraph of $\widehat{G}$.
    Let $v$ be a vertex in $V(\widehat{G})\setminus V(G)$.
    By~\cref{lem:bouchet3Ways}, at least one of $\widehat{G}-v$, $\widehat{G}*v-v$, and $\widehat{G}/v$ contains $G'$ as a vertex-minor. 

    If $\widehat{G}-v$ contains $G'$ as a vertex-minor, then $G'$ is a $(t-1)$-perturbation of~$G$, implying our conclusion by the induction hypothesis. 
    
    If $\widehat{G}*v-v$ contains $G'$ as a vertex-minor, then  
    let $G_1=(\widehat{G}*v)[V(G)]$. Then $G'$  is a $(t-1)$-perturbation of $G_1$ and by the induction hypothesis, $G'$ is a rank-$(2(t-1))$ perturbation of $G_1$.
    Note that $G_1$ is a rank-$1$ perturbation of $G$, and therefore $G'$ is a rank\nobreakdash-$(2t-1)$ perturbation of $G$.

    It remains to consider the case that $\widehat{G}/v$ contains $G'$ as a vertex-minor.
    We may assume that $v$ has a neighbor~$w\in V(G)$ in $\widehat{G}$, because otherwise 
    $\widehat{G}/v$ contains $G$ as an induced subgraph, implying that 
    $G'$ is a $(t-1)$-perturbation of~$G$.
    Let $G_2=(\widehat{G}\pivot vw)[V(G)]$.
    Observe that $G'$ is a $(t-1)$-perturbation of~$G_2$
    and therefore by the induction hypothesis, $G'$ is a rank\nobreakdash-$(2(t-1))$ perturbation of~$G_2$.
    Observe that
    for $S \coloneqq N_{\widehat G}(v)\cap V(G)$ and $T \coloneqq N_{G}(w)\cup \{w\}$, 
    \[ 
    A(G_2)-A(G) = \bordermatrix{
       & T\setminus S & S\cap T  & S\setminus T & V(G)\setminus (S\cup T) \cr   
       T\setminus S 
       & 0 & 1 & 1 & 0\cr  
       S\cap T
       & 1 & 0 & 1 & 0 \cr 
       S\setminus T  & 0 & 1 & 1 & 0 \cr 
       V(G)\setminus (S\cup T)
        & 0 & 0& 0& 0 
    }, 
    \] 
    and therefore  $G_2$ is a rank-$2$ perturbation of~$G$.
    Thus, $G'$ is a rank-$2t$ perturbation of~$G$.
\end{proof}



One may ask whether we can replace the operation of ``taking a $t$-perturbation'' with the operation of ``taking a $0$-perturbation and removing $t$ vertices'' in \cref{thm:main}. In other words, 
is it true that for any circle graph $H$ and integer $k$, there exists an integer $t = t(k, H)$ so that for every graph $G$, either $G$ has a vertex-minor isomorphic to $kH$, or there exist a graph $\widetilde{G}$ locally equivalent to $G$ and a set $X$ of at most $t$ vertices in $G$ such that $\widetilde{G}-X$ has no vertex-minor isomorphic to~$H$?
We show that this variant does not hold for some small graph $H$.  Let $P_n$ denote the path graph on $n$ vertices.

\begin{lemma}
\label{lem:pertNecessary}
    For every non-negative integer $t$, there exists a graph $G_t$ such that $G_t$ has no vertex-minor isomorphic to $2P_4$, but for every $0$-perturbation $\widetilde{G_t}$ of $G_t$ and set $X$ of at most~$t$ vertices in $\widetilde{G_t}$, the graph $\widetilde{G_t}-X$ has a vertex-minor isomorphic to $P_4$. 
\end{lemma}
\begin{proof}
    We show that $G_t$ can be taken as a complete multipartite graph. To see this, we now describe the graphs that are locally equivalent to complete multipartite graphs. Intuitively, when we locally complement, each part of the complete multipartite graph becomes either an independent set, a clique, or a star whose low-degree vertices are pendant in $G$ (that is, have degree one in $G$). The ``quotient'' graph is either a clique or a star. 

    For two finite families $\mathcal{A}$ and $\mathcal{B}$ of sets  (possibly with $\mathcal{A}=\emptyset$ or $\mathcal{B}=\emptyset$) where all sets in~$\mathcal{A}\cup \mathcal{B}$ are pairwise disjoint, 
    let $\mathcal{C}(\mathcal{A}, \mathcal{B})$ be a family of graphs $G$ on $\bigcup_{X\in \mathcal{A} \cup \mathcal{B}}X$ such that 
    \begin{itemize}\setlength\itemsep{0em}
        \item every $A\in \mathcal{A}$ is independent in $G$, 
        \item every $B\in \mathcal{B}$ induces a star graph with center, say $c_B\in B$, such that every vertex in $B\setminus \{c_B\}$ has degree one in $G$, and
        \item any two vertices in distinct members of $\mathcal{A}\cup \{\{c_B\}:B\in \mathcal{B}\}$ are adjacent.
    \end{itemize}
    Note that if $\mathcal{B}=\emptyset$, then $G$ is a complete multipartite graph.
    
    We need one more construction. For a set $S$ and two finite families $\mathcal{A}$ and $\mathcal{B}$ of sets (possibly with $\mathcal{A}=\emptyset$ or $\mathcal{B}=\emptyset$) where all sets in $\mathcal{A}\cup \mathcal{B}\cup \{S\}$ are pairwise disjoint,   let $\mathcal{D}(S, \mathcal{A}, \mathcal{B})$ be the family of graphs $G$ on $S\cup (\bigcup_{X\in \mathcal{A} \cup \mathcal{B}}X)$ such that 
    \begin{itemize}  \setlength\itemsep{0em}
        \item $S$ is a clique or an independent set  in $G$,
        \item every $A\in \mathcal{A}$ is a clique in $G$, 
        \item every $B\in \mathcal{B}$ induces a star graph with center, say $c_B\in B$, such that every vertex in $B\setminus \{c_B\}$ has degree one in $G$, and
        \item there are no edges between any two distinct sets in $\mathcal{A}\cup \{\{c_B\}:B\in \mathcal{B}\}$, and 
        \item each vertex in $S$ is adjacent to all vertices in any member of $ \mathcal{A}\cup \{\{c_B\}:B\in \mathcal{B}\}$.
    \end{itemize}
    We claim that the family of all possible graphs of these two types is closed under taking local complementations. 
    
    Suppose that $G\in \mathcal{C}(\mathcal{A}, \mathcal{B})$ for some families $\mathcal{A}$ and $\mathcal{B}$, and let $v\in V(G)$. 
    If $v\in A$ for some $A\in \mathcal{A}$, then
    $G*v\in \mathcal{D}(A, \mathcal{A}\setminus \{A\}, \mathcal{B})$. So assume that $v\in B$ for some $B\in \mathcal{B}$. If $v$ has degree~$1$ in~$G$, then $G*v=G$. So, we may assume $v=c_B$. Then $G*v\in \mathcal{D}(B, \mathcal{A}, \mathcal{B}\setminus \{B\})$. 

    Now, suppose that $G\in \mathcal{D}(S,\mathcal{A}, \mathcal{B})$ for some set $S$ and families $\mathcal{A}$ and $\mathcal{B}$, and let $v\in V(G)$. 
    If $v\in A$ for some $A\in \mathcal{A}$, then $G*v\in \mathcal{D}(S, \mathcal{A}\setminus \{A\}, \mathcal{B}\cup \{A\})$. 
    We may assume that $G*v\neq G$ and therefore  $v$ has degree at least $2$. 
    If $v\in B$ for some $B\in \mathcal{B}$, then 
    $v=c_B$, and therefore $G*v\in \mathcal{D}(S, \mathcal{A}\cup \{B\}, \mathcal{B}\setminus \{B\})$. 
    
    Lastly, assume that $v\in S$. If $S$ is independent in $G$, then $G*v\in \mathcal{C}(\mathcal{A}\cup \{S\}, \mathcal{B})$. If $S$ is a clique in $G$, then $G*v\in \mathcal{C}(\mathcal{A}, \mathcal{B}\cup \{S\})$.
    This proves the claim. 

    It can be observed that no graph in $\mathcal{C}(\mathcal{A}, \mathcal{B})$ or $\mathcal{D}(S, \mathcal{A}, \mathcal{B})$ contains an induced subgraph isomorphic to $2P_4$, as a complete graph or a star graph has no $P_4$ and $P_4$ has no twins. 
    Therefore $G_t \coloneqq K_{t+2,t+2,\ldots,t+2}$ (with $t+3$ parts) has no vertex-minor isomorphic to $2P_4$.
    
    On the other hand, we claim that if $G$ is a graph in $\mathcal{C}(\mathcal{A}, \mathcal{B})$ where $\abs{\mathcal{A}\cup \mathcal{B}}\ge t+3$ and each set in $\mathcal{A}\cup \mathcal{B}$ has size at least $t+2$, then for every graph $\widetilde{G}$ locally equivalent to $G$ and every set $X$ of size at most $t$, $\widetilde{G}-X$ still contains a vertex-minor isomorphic to~$P_4$. 
    If $\widetilde{G}\in \mathcal{C}(\mathcal{A}', \mathcal{B}')$ for some $\mathcal{A}', \mathcal{B}'$, then there are two sets $U, W$ in $\mathcal{A}'\cup \mathcal{B}'$ that do not intersect $X$ and $\widetilde{G}[U\cup W]$ is a graph in $\mathcal{C}(\mathcal{A}'', \mathcal{B}'')$ for some $\mathcal{A}''$, $ \mathcal{B}''$ having at least two parts in total, and it contains a vertex-minor isomorphic to $P_4$. 
    A similar argument holds when $\widetilde{G}\in \mathcal{D}(S', \mathcal{A}', \mathcal{B}')$ for some $S'$, $\mathcal{A}'$, $\mathcal{B}'$, where in this case, we use the fact that $S'\setminus X$ is still a non-empty set complete to the rest of the graph, and $\abs{\mathcal{A}'\cup \mathcal{B}'}\ge 2$. This completes the proof of the second claim.
    Thus, for every $0$-perturbation $\widetilde G_t$ of $G_t$, $\widetilde G_t-X$ for every set $X$ of at most $t$ vertices contains a vertex-minor isomorphic to $P_4$. 
\end{proof}

\section{Pivot-minors and matroids}
\label{sec:pivotminor}

For bipartite graphs, our results can be generalized to a stricter notion of containment: \emph{pivot-minors}. (Recall from \cref{sec:prelims} that every pivot-minor is a vertex-minor but not necessarily vice versa.)
In \cref{sec:pivots}, we prove that bipartite circle graphs exhibit an Erd\H{o}s-P\'{o}sa property with respect to pivot-minors (\cref{thm:mainPivot}).
In \cref{sec:matroids}, we show that this implies a similar statement (\cref{cor:binaryMatroids}) about binary matroids, of which we also prove a rough converse (\cref{prop:binMatConverse}).
This section discusses the connection between pivot-minors in bipartite graphs and minors of binary matroids and states a few initial theorems.


Given a multigraph $G$, we write $M(G)$ for the cycle matroid of $G$. Let $M$ be a matroid. We write $E(M)$ for the set of elements of $M$ and $M^*$ for the dual matroid of $M$. Given $e \in E(M)$, we write $M/e$ for the matroid obtained from $M$ by contracting $e$ and $M\setminus e$ for the matroid obtained from $M$ by deleting $e$. Given a positive integer $k$, we write $kM$ for the matroid obtained by the union of $k$ disjoint copies of $M$. So for any multigraph~$G$, the matroids $M(kG)$ and $kM(G)$ are isomorphic. We say that a matroid $M$ is \emph{represented by} a matrix $A$ if the columns of $A$ are indexed by $E(M)$, and a subset $I$ of~$E(M)$ is independent in $M$ if and only if the columns of $A$ corresponding to~$I$ are linearly independent. 

In order to obtain \cref{cor:binaryMatroids}, we need to ``keep track'' of the sides of bipartite graphs. So, for the rest of this paper, a \emph{bipartite graph} is a tuple $\mathcal{G}=(G, A, B)$ so that $G$ is a graph and $A$ and $B$ are disjoint sets of vertices whose union is $V(G)$ such that every edge of~$G$ has exactly one end in $A$ and one end in $B$. Note that for any edge $uv$ of $G$ with $u \in A$, the tuple $(G \times uv, (A\setminus \{u\})\cup \{v\}, (B\setminus \{v\})\cup \{u\})$ is again a bipartite graph; we denote it by~$\mathcal{G}\times uv$.
For bipartite graphs, this is what it means to \emph{pivot} on an edge $uv$. We say that two bipartite graphs $\mathcal{G}=(G,A,B)$ and $\mathcal{G}'=(G', A', B')$ are \emph{strongly isomorphic} if there exists an isomorphism between $G$ and $G'$ which sends $A$ to $A'$ and $B$ to $B'$. Finally, the \emph{bipartite adjacency matrix} of $\mathcal{G}=(G,A,B)$ is the matrix with rows indexed by $A$ and columns indexed by $B$ so that an entry is $1$ if those vertices are adjacent in $G$ and $0$ otherwise.

With this set-up, pivot-minors of bipartite graphs are equivalent to minors of binary matroids. We need some more definitions to state this result carefully. 
Given a binary matroid $M$ and a base~$B$ of~$M$, the \emph{fundamental graph} $\mathcal{F}(M,B)$ is the bipartite graph $(G, B, E(M)\setminus B)$ where each element $e \in E(M)\setminus B$ is adjacent to the other elements in its fundamental circuit with respect to~$B$. The \emph{fundamental circuit} of an element $e\in E(M) \setminus B$ with respect to $B$ is the unique circuit of $M$ which is contained in $B \cup \{e\}$. Note that if $M$ is represented by a matrix $\big[ I \mid A \big]$ over $\GF(2)$ where $I$ is an identity matrix whose columns and rows are indexed by $B$, then $A$ is the bipartite adjacency matrix of~$G$. Thus, every bipartite graph is the fundamental graph of some binary matroid $M$ and base $B$ of~$M$.

\begin{lemma}[Bouchet~\cite{graphicIsoSystems}; see~\cite{RWAndVM}]
\label{lem:pivotFundMatroid}
    Let $M_1$ and $M_2$ be binary matroids with respective bases~$B_1$ and $B_2$.
    Then $M_1$ is (isomorphic to) a minor of $M_2$ if and only if $\mathcal{F}(M_1,B_1)$ is (strongly isomorphic to) a pivot-minor of $\mathcal{F}(M_2,B_2)$.
\end{lemma}

Moreover, de Fraysseix~\cite{VMandInterlacement} proved that fundamental graphs of planar multigraphs are exactly the bipartite circle graphs. (A bipartite graph $\mathcal{G} = (G, A, B)$ is a \emph{circle graph} if $G$ is a circle graph. A \emph{spanning forest} of a multigraph $G$ is a maximal acyclic subgraph of~$G$.)

\begin{theorem}[de Fraysseix~{\cite[Proposition 6]{VMandInterlacement}}] 
\label{thm:planarBipCircle}
    Let $\mathcal{G}=(G, A, B)$ be a bipartite graph. Then $\mathcal{G}$ is a circle graph if and only if there exists a planar multigraph $H$ and a spanning forest~$T$ of~$H$ so that $\mathcal{G}$ is strongly isomorphic to~$\mathcal{F}(M(H),E(T))$.
\end{theorem}

We require the following grid theorem for pivot-minors in bipartite graphs. This theorem is a corollary of \cref{lem:pivotFundMatroid} and the grid theorem for $\GF(q)$-representable matroids of Geelen, Gerards, and Whittle~\cite{GridThmRepresentable}. This connection was discovered by Oum~\cite{RWAndVM}, who conjectures~\cite[Conjecture 1.1]{lineGraphsGridThm} that for any bipartite circle graph $H$ (without specified sides), every graph $G$ of sufficiently large rank-width contains a pivot-minor isomorphic to $H$. For reasons we do not get into having to do with matroid duality, for the next theorem, it does not matter whether there are specified sides or not; it can be shown that the two formulations are equivalent.

\begin{theorem}[Geelen, Gerards, and Whittle~\cite{GridThmRepresentable}]
\label{thm:gridPivot}
    For any bipartite circle graph $\mathcal{H}$, there exists an integer $r_1(\mathcal{H})$ so that every bipartite graph $\mathcal{G}$ with no pivot-minor strongly isomorphic to $\mathcal{H}$ has rank-width at most~$r_1(\mathcal{H})$.
\end{theorem}

We also require the following analog of \cref{thm:wqo} about well-quasi-ordering. This theorem is a corollary of \cref{lem:pivotFundMatroid} and a theorem of Geelen, Gerards, and Whittle~\cite{branchWidthWQO} about $\GF(q)$-representable matroids of large branch-width. (The branch-width of a binary matroid is related to the rank-width of its fundamental graph; see~\cite{RWAndVM}.) We note that \cref{thm:wqoPivot} is a corollary of \cref{thm:wqo} too, as explained in~\cite{RWandWQO}.

\begin{theorem}[Geelen, Gerards, and Whittle~\cite{branchWidthWQO}]
\label{thm:wqoPivot}
    For any integer $r$ and any infinite set $\mathcal{C}$ of bipartite graphs of rank-width at most~$r$, there exist distinct bipartite graphs $\mathcal{G}_1, \mathcal{G}_2 \in \mathcal{C}$ so that $\mathcal{G}_1$ is strongly isomorphic to a pivot-minor of $\mathcal{G}_2$.
\end{theorem}

\section{An Erd\H{o}s-P\'{o}sa property for pivot-minors}
\label{sec:pivots}

In this section, we adjust the steps of the main result to prove that bipartite circle graphs exhibit an Erd\H{o}s-P\'{o}sa property with respect to pivot-minors (\cref{thm:mainPivot}).
To this end, we need some analogous definitions of perturbations and robustness. So, given a non-negative integer~$t$, we say that a bipartite graph $\mathcal{G}_1$ is a \emph{$t$-pivot-perturbation} of a bipartite graph $\mathcal{G}_2$ with the same vertex set as $\mathcal{G}_1$ if there exists a bipartite graph which has at most $t$ more vertices and contains both $\mathcal{G}_1$ and $\mathcal{G}_2$ as pivot-minors. Likewise, we say that a bipartite graph $\mathcal{G}$ is \emph{$t$-pivot-robust} for a bipartite graph $\mathcal{H}$ if every $t$-pivot-perturbation of $\mathcal{G}$ contains a pivot-minor that is strongly isomorphic to $\mathcal{H}$. 

\Cref{lem:pertProperties,lem:robustCommute} readily translate to this setting. As such, we give the analogous statements below without proof. 

\begin{lemma}
    \label{lem:pertPropertiesPivot}
    Let $s$, $t$ be non-negative integers and $\mathcal{G}_1$, $\mathcal{G}_2$, $\mathcal{G}_3$ be bipartite graphs. 
    If $\mathcal{G}_2$ is an $s$-pivot-perturbation of~$\mathcal{G}_1$ and $\mathcal{G}_3$ is a $t$-pivot-perturbation of $\mathcal{G}_2$, then $\mathcal{G}_3$ is an $(s+t)$-pivot-perturbation of~$\mathcal{G}_1$.
\end{lemma}

\begin{lemma}
    \label{lem:robustCommutePivot}
    Let $\mathcal{H}$ be a bipartite graph, and let $r$ and $t$ be non-negative integers with $t \geq r$. If $\mathcal{G}$ is a bipartite graph that is $t$-pivot-robust for $\mathcal{H}$, then any $r$-pivot-perturbation of~$\mathcal{G}$ is $(t-r)$-pivot-robust for~$\mathcal{H}$.
\end{lemma}

For an analogous statement to \cref{lem:pertCutRank} we need to be careful because our definition of a $t$-pivot-perturbation requires us to find a \emph{bipartite} graph containing two bipartite graphs as pivot-minors.
\begin{lemma}
    \label{lem:pertCutRankPivot}
    For any bipartite graph $\mathcal{G}=(G, A, B)$ and any subset $X$ of $V(G)$, the bipartite graph $(G-\delta_G(X), A, B)$ is a $2\rho_G(X)$-pivot-perturbation of~$\mathcal{G}$.
\end{lemma}
\begin{proof}
    For disjoint subsets $S$ and $T$ of $V(G)$, let us write $\rho_G(S,T)$ for the rank of the $S\times T$ submatrix of the adjacency matrix of $G$.
    Let $Y=V(G)\setminus X$.
    Let $r=\rho_G(X)$, $r_1=\rho_G(X\cap A,Y\cap B)$, and $r_2=\rho_G(X\cap B,Y\cap A)$. 
    Since $G$ is bipartite, we have that $r=r_1+r_2$ and 
    there are no edges between $X\cap A$ and $Y\cap A$
    and no edges between $X\cap B$ and $Y\cap B$.    

    By a well-known property in linear algebra, for each $i\in\{1,2\}$, every matrix of rank~$r_i$ is a sum of
     $r_i$ rank-one matrices. 
    Therefore, like the proof of~\cref{lem:pertCutRank}, 
    there exist nonempty $X_1, X_2, \ldots, X_r\subseteq X$ and $Y_1, Y_2, \ldots, Y_r\subseteq Y$ such that 
    \begin{itemize}
        \item $X_1,X_2,\ldots,X_{r_1}\subseteq X\cap A$,
        \item $Y_1,Y_2,\ldots,Y_{r_1}\subseteq Y\cap B$,
        \item $X_{r_1+1},X_{r_1+2},\ldots,X_r\subseteq X\cap B$,
        \item $Y_{r_1+1},Y_{r_1+2},\ldots,Y_r\subseteq Y\cap A$, and 
        \item a vertex $x\in X$ is adjacent to a vertex $y\in Y$ if and only if 
    there are an odd number of $i\in [r]$ such that $(x,y)\in X_i\times Y_i$.
    \end{itemize}
    
    Let us construct a graph $\widehat G$ by adding $2r$ vertices $x_1$, $x_2$, $\ldots$, $x_r$, $y_1$, $y_2$, $\ldots$, $y_r$ to $G$ and making $x_i$ adjacent to all vertices in $X_i$ and $y_i$ adjacent to all vertices in $Y_i\cup \{x_i\}$.
    Then $(\widehat G,A\cup \{y_1,y_2,\ldots,y_{r_1},x_{r_1+1},x_{r_1+2},\ldots,x_r\}, B\cup\{x_1,x_2,\ldots,x_{r_1},y_{r_1+1},y_{r_1+2},\ldots,y_{r}\})$ is a bipartite graph and 
    \[ G-\delta_G(X)=\widehat G\pivot x_1y_1\pivot x_2y_2\pivot\cdots\pivot x_ry_r -\{x_1,x_2,\ldots,x_r,y_1,y_2,\ldots,y_r\}.\]  
    Therefore $(G-\delta_G(X),A,B)$ is a $2r$-pivot-perturbation of~$\mathcal G$.
\end{proof}

We do not know whether a direct translation of \cref{lem:vm-perturbation} is true for pivot-minors. However, we can prove a weaker lemma (\cref{lem:vm-perturbationPivot}), which is sufficient for our purposes.
For convenience, given a bipartite graph $\mathcal{G}=(G, A, B)$ and a set $X \subseteq V(G)$, we write $\mathcal{G}[X]$ for the bipartite graph $(G[X], A \cap X, B \cap X)$.
We call $\mathcal{G}[X]$ the subgraph of $\mathcal{G}$ \emph{induced on}~$X$.

We use a well-known matrix interpretation of pivoting; see, for instance, Moffatt~\cite{deltaMatroidSurvey}, or Geelen~\cite[Theorem~2.5]{generalizedTU} for details. Let \[ 
    M = \bordermatrix{
       & Y & V\setminus Y  \cr   
       Y & \alpha & \beta \cr  
       V\setminus Y & \gamma & \delta 
    }
    \] be a $V\times V$ matrix over a field. If $\alpha=M[Y]$ is nonsingular, then we define 
     \[ 
    M*Y = \bordermatrix{
       & Y & V\setminus Y  \cr   
       Y & \alpha^{-1} & \alpha^{-1}\beta \cr  
       V\setminus Y & -\gamma \alpha^{-1} & \delta-\gamma\alpha^{-1}\beta
    }.
    \]
    It is well known that if $G'$ is obtained from $G$ by a sequence of pivoting edges and $Y$ is the set of all vertices that are incident with an odd number of pivoted edges, then $M'=M*Y$, where we write $M$ and $M'$ for the adjacency matrices of $G$ and $G'$, respectively. (We always view adjacency matrices as being over $\GF(2)$.) Conversely, if $G$ is a graph with adjacency matrix $M$ and $X$ is a set of vertices in $G$, then $M*X$ is the adjacency matrix of a graph that is pivot-equivalent to $G$.

\begin{lemma}
\label{lem:vm-perturbationPivot}
    Let $\mathcal{G}=(G,A,B)$ be a bipartite graph, and let $X$ be a set of vertices of $G$. If $\mathcal{G}'=(G', A', B')$ is a bipartite graph which can be obtained from $\mathcal{G}$ by successively pivoting on edges with neither end in $X$, then $\mathcal{G}'[X]$ is a $\rho_G(X)$-pivot-perturbation of $\mathcal{G}[X]$.
\end{lemma}
\begin{proof}
    Let $M$ and $M'$ denote the adjacency matrices of $G$ and $G'$, respectively. By taking the set of all vertices which are incident with an odd number of pivoted edges, we obtain a set $Y \subseteq V(G)\setminus X$ so that, if  
    \[ 
    M = \bordermatrix{
        & Y\cap A & Y\cap B & A\setminus Y & B\setminus Y  \cr   
        Y\cap A& 0 & \alpha & 0 & \beta \cr
        Y\cap B & \alpha^T & 0 & \gamma & 0 \cr
        A\setminus Y & 0 & \gamma^T & 0 & \delta \cr
        B\setminus Y & \beta^T & 0 & \delta^T & 0 
    },
    \] then
    \[ 
    M' = M*Y = \bordermatrix{
        & Y\cap A & Y\cap B & A\setminus Y & B\setminus Y  \cr   
        Y\cap A& 0 & (\alpha^{-1})^T & (\alpha^{-1})^T \gamma & 0 \cr
        Y\cap B & \alpha^{-1} & 0 & 0 & \alpha^{-1}\beta\cr
        A\setminus Y & -\gamma^T \alpha^{-1} & 0 & 0 & \delta -\gamma^T\alpha^{-1} \beta  \cr
        B\setminus Y & 0 & -\beta^T(\alpha^{-1})^T & \delta^T-\beta^T(\alpha^{-1})^T \gamma  & 0 
    }.
    \] 
    We may assume that $X\cup Y=V(G)$, because deleting vertices can not increase $\rho_G(X)$.

    Observe that $\rho_G(X)=\rank(\beta)+\rank(\gamma)$. 
    By symmetry between $A$ and $B$, we may assume that $\rank(\beta)\le \rank(\gamma)$.
    Let $r$ be the rank of $\gamma^T\alpha^{-1}\beta$, which is equal to the rank of its transpose $\beta^T(\alpha^{-1})^T\gamma$.
    Then $r\le \rank(\beta)$.
    By using the proof argument of \cref{lem:rank}, 
    we can write $\gamma^T\alpha^{-1}\beta$ as the sum of $r$ rank-$1$ matrices $M_1$, $M_2$, $\ldots$, $M_r$ over $\GF(2)$.
    For each $i\in [r]$, let $A_i$ be a subset of $A\setminus Y$ and $B_i$ be a subset of $B\setminus Y$ such that for all $a\in A\setminus Y$ and $b\in B\setminus Y$, an $(a,b)$-entry of $M_i$ is $1$ if and only if $a\in A_i$ and $b\in B_i$. 
    Now we construct a new bipartite graph $\widehat{\mathcal{G}}=(\widehat G,A\cup\{a_1,a_2,\ldots,a_r\}, B\cup \{b_1,b_2,\ldots,b_r\})$ as follows;
    \begin{itemize}
        \item $\widehat{\mathcal{G}}[X]=\mathcal G[X]$, 
        \item For each $i\in [r]$, 
    $a_i$ and $b_i$ are adjacent, 
    $a_i$ is adjacent to all vertices in $B_i$, and $b_i$ is adjacent to all vertices in $A_i$.
    \end{itemize}
    Then $G'[X]= \widehat G\pivot a_1b_1\pivot a_2b_2\cdots\pivot a_rb_r-\{a_1,a_2,\ldots,a_r,b_1,b_2,\ldots,b_r\}$.
    Therefore $\mathcal{G}'[X]$ is a $2r$-pivot-perturbation of $\mathcal{G}[X]$
    and $2r\le \rho(G)$.
\end{proof}

Next, we note that since \cref{lem:manyparts} only uses \cref{lem:pertProperties,lem:robustCommute,lem:pertCutRank}, we can prove a statement which is analogous to \cref{lem:manyparts} (but for pivot-minors of bipartite graphs) by applying \cref{lem:pertPropertiesPivot,lem:robustCommutePivot,lem:pertCutRankPivot} in their respective places. 

\begin{lemma}\label{lem:manypartsPivot}
    Let $k$, $t$, and $r$ be positive integers, and let $\mathcal H$ be a bipartite graph with no isolated vertices and with components $\mathcal H_1$, $\mathcal H_2$, $\ldots$, $\mathcal H_m$.
    Then for every bipartite graph $\mathcal G=(G,A,B)$ and every rank-decomposition~$T$ of~$G$ of width at most $r$, at least one of the following holds.
    \begin{enumerate}
    \item There exists $i \in [m]$ such that there is a $(4rmk+2tmk)$-pivot-perturbation of~$\mathcal G$ which has no pivot-minor strongly isomorphic to~$\mathcal H_i$.
    \item There exist pairwise vertex-disjoint subtrees $(T_{i,j}:i \in [m], j \in [k])$ of~$T$ such that for all $i\in [m]$ and $j\in [k]$, the graph $G(T_{i,j})$ is $t$-pivot-robust for $\mathcal H_i$.
    \end{enumerate}
\end{lemma}

Then we can prove a statement which is analogous to \cref{prop:chain} by applying \cref{thm:wqoPivot,lem:vm-perturbationPivot} instead of \cref{thm:wqo} and \cref{lem:vm-perturbation}, respectively. It is essential to notice that, in the proof, when we obtain $G_{i,j+1}$ from $G_{i,j}$, we are only pivoting on edges with both ends in $Y_{i,j+1}$.

\begin{proposition}\label{prop:chainPivot}
    Let $k$ be a positive integer, let $t$, $r$ be non-negative integers, and let $\mathcal H$ be a bipartite graph with no isolated vertices and with components $\mathcal H_1$, $\mathcal H_2$, $\ldots$, $\mathcal H_m$. 
    Then there exist constants $c = c_{\mathcal H}(t,r)$ and $p=p_{\mathcal H}(k,t,r)$ such that for every bipartite graph $\mathcal G$ of rank-width at most~$r$, at least one of the following holds.
    \begin{enumerate}
    \item There exists $i \in [m]$ such that there is a $p$-pivot-perturbation of $\mathcal G$ which has no pivot-minor strongly isomorphic to $\mathcal H_i$.
    \item There exists a pivot-minor $\widetilde{\mathcal G}$ of $\mathcal G$ whose vertex set is the union of pairwise disjoint sets $(X_{i,j}: i \in [m], j \in [k])$ such that for all $i \in [m]$ and $j \in [k]$, we have $\abs{X_{i,j}}\leq c$, the graph $\widetilde{\mathcal G}[X_{i,j}]$ is $t$-pivot-robust for $\mathcal H_i$, and $\rho_{\widetilde{\mathcal G}}(X_{i,j})\leq r(m^2+1)$.
\end{enumerate}
\end{proposition}

Finally, notice that in the proof of \cref{lem:cleanupchain}, we only locally complement when there is a large clique. So we never locally complement in bipartite graphs, which are closed under taking pivot-minors. 

\begin{lemma}\label{lem:cleanupchainPivot}
    Let $c$, $k$, and $\ell$ be integers with $0\le \ell < \binom{c}{2}+c$. Let $\mathcal G$ be a bipartite graph, and let $\mathcal{X}=(X_1, X_2, \ldots, X_{R_4(\inside)})$ be a $c$-uniform chain of length $R_4(\inside)$ in $G$ such that the first~$\ell$ pairs in $\{(j_1,j_2):1\leq j_1\leq j_2\leq c\}$ with respect to the lexicographic order are fixed. Then 
    there exist a bipartite graph $\widetilde{\mathcal G}$ and a subchain $\mathcal{Y}$ of $\mathcal{X}$ of length $k$ such that
    \begin{itemize}
        \item $\widetilde{\mathcal G}$ is obtained from $\mathcal G$ by applying pivots on edges with both ends in $V(\mathcal{X})\setminus V(\mathcal{Y})$,
        \item the first $\ell+1$ pairs in the lexicographic order are fixed.
    \end{itemize}
\end{lemma}

Now we obtain the following proposition as a corollary of \cref{lem:cleanupchainPivot}.
Given a bipartite graph $\mathcal{H}$ and a positive integer $k$, we write $k\mathcal{H}$ for the bipartite graph which is obtained by taking $k$ disjoint copies of $\mathcal{H}$, where we take the union of the respective sides.

\begin{proposition}\label{prop:cleaningChainsPivot}
    For any positive integers $c$ and $k$, and any bipartite graph $\mathcal H$ with components $\mathcal H_1$, $\mathcal H_2$, $\ldots$, $\mathcal H_m$, there exists an integer $K = K_{\mathcal H}(c,k)$ so that the following holds. Let $\mathcal G$ be a bipartite graph whose vertex set is the union of pairwise disjoint sets $(X_{i,j}: i\in [m], j \in [K])$ such that for all $i \in [m]$ and $j \in [k]$, the set $X_{i,j}$ has size at most $c$, the bipartite graph $\mathcal G[X_{i,j}]$ is $r(m^3+m)$-pivot-robust for $\mathcal H_i$, and $\rho_{\mathcal G}(X_{i,j})\leq r(m^2+1)$. Then $\mathcal G$ contains a pivot-minor strongly isomorphic to $k\mathcal H$.
\end{proposition}
\begin{proof}
    For each $j\in [K]$, there are at most $cm$ possibilities for $\sum_{i \in [m]} \abs{X_{i,j}}$.
    So by the pigeonhole principle, and thinking of $K$ as any sufficiently large integer, we may assume that this sum of sizes is the same for every $j \in [K]$. Now, for each $j\in [K]$, we set $X_j = \bigcup_{i \in [m]}X_{i,j}$, and we arbitrarily order $X_j$. Let $\mathcal{X}=(X_1, X_2, \ldots, X_K)$. Then $\mathcal{X}$ is a chain of length~$K$. 
    
    Note that $\mathcal{X}$ is $C$-uniform for some positive integer $C \leq mc$. Also note that for every $i\in [m]$ and $j\in [K]$, the graph $\mathcal G[X_{j}]$ is $r(m^3+m)$-pivot-robust for $\mathcal H_i$ because $X_{i,j}\subseteq X_j$. 
    Moreover, for every $j\in [K]$, we have $\rho_{\mathcal G}(X_{j})\le r(m^3+m)$ by the submodularity of the cut-rank function. By applying \cref{lem:cleanupchainPivot} at most $\binom{C}{2}+C$ times to~$\mathcal X$, we obtain a bipartite graph~$\widetilde{\mathcal G}$ that is pivot-equivalent to~$\mathcal G$ 
    and a subchain $\mathcal{Y}=(Y_1,Y_2,\ldots,Y_{km})$ of~$\mathcal{X}$ of length~$km$ 
    such that every pair $(j_1, j_2)$ with $1\le j_1\le j_2\le C$ is fixed with respect to~$\mathcal{Y}$ and~$\widetilde{\mathcal G}$.
    This means that $Y_j$ has no edges to $Y_{j'}$ for any distinct $j,j' \in [km]$.
    For each $j\in [km]$, the graph $\widetilde{\mathcal G}[Y_j]$ is an $r(m^3+m)$-pivot-perturbation of $\mathcal G[Y_j]$ by \cref{lem:vm-perturbationPivot}.
    Thus, by \cref{lem:robustCommutePivot}, the graph~$\widetilde{\mathcal G}[Y_j]$ is $0$-pivot-robust for $\mathcal H_i$ for each $i \in [m]$.
    In particular, $\widetilde{\mathcal G}[Y_j]$ has each of $\mathcal H_1$, $\mathcal H_2$, $\ldots$, $\mathcal H_m$ as a pivot-minor.
    It follows that $\widetilde{\mathcal G}$, and therefore $\mathcal G$, contains a pivot-minor strongly isomorphic to $k\mathcal H$.
\end{proof}

The rest of the proof proceeds in exactly the same way as the one for \cref{thm:main}, yielding the following theorem.

\begin{theorem}
\label{thm:mainPivot}
    For any bipartite circle graph $\mathcal{H}$ with at least one edge and for any positive integer $k$, there exists an integer $p_1=p_1(k,\mathcal{H})$ so that each bipartite graph $\mathcal{G}$ either has a pivot-minor strongly isomorphic to $k\mathcal{H}$, or has a $p_1$-pivot-perturbation $\widetilde{\mathcal{G}}$ which has no pivot-minor strongly isomorphic to $\mathcal{H}$.
\end{theorem}
\begin{proof}
    First, suppose that $\mathcal H$ has $m$ isolated vertices for some $m>0$. Let $\mathcal H'$ be the graph obtained from $\mathcal H$ by removing all isolated vertices. We may take $t(k,\mathcal H)\coloneqq t(k+km,\mathcal H')$.

    Therefore, we may assume that $\mathcal H$ has no isolated vertices. Let $\mathcal G$ be a bipartite graph; we may assume that $\mathcal G$ has no pivot-minor strongly isomorphic to~$k\mathcal H$. 
    By~\cref{thm:gridPivot}, there exists $r\coloneqq r_1(k\mathcal H)$ such that every bipartite graph with no pivot-minor strongly isomorphic to~$k\mathcal H$ has rank-width at most $r$. 
    (Note that $k\mathcal H$ is a bipartite circle graph since circle graphs are closed under taking disjoint unions.) Thus, the rank-width of~$\mathcal G$ is at most~$r$.

    Let $\mathcal H_1$, $\mathcal H_2$, $\ldots$, $\mathcal H_m$ be the components of $\mathcal H$. 
    Let $c_{\mathcal H}$ and $p_{\mathcal H}$ be the functions given by \cref{prop:chainPivot}, and let $K_{\mathcal H}$ be the function given by \cref{prop:cleaningChainsPivot}. Set $c \coloneqq c_{\mathcal H}(r(m^3+m),r)$, $K\coloneqq K_{\mathcal H}(c,k)$, and $p_1 \coloneqq p_{\mathcal H}(K,r(m^3+m), r)$. By \cref{prop:chainPivot}, at least one of the following holds.
     \begin{enumerate}
    \item There exists $i \in [m]$ such that there is a $p_1$-pivot-perturbation of~$\mathcal G$ which has no pivot-minor strongly isomorphic to $\mathcal H_i$.
    \item There exist 
    a pivot-minor $\widetilde{\mathcal G}$ of~$\mathcal G$ whose vertex set is the union of pairwise disjoint sets $(X_{i,j}: i\in [m], j \in [K])$ such that for all $i\in [m]$ and $j\in [K]$, we have $\abs{X_{i,j}}\le c$, the graph $\widetilde{\mathcal G}[X_{i,j}]$ is $r(m^3+m)$-pivot-robust for $\mathcal H_i$, 
    and $\rho_{\widetilde{\mathcal G}}(X_{i,j})\le r(m^2+1)$. 
    \end{enumerate}
    In the first case, we have found the desired pivot-perturbation and are done. In the second case, $\mathcal G$ contains a pivot-minor strongly isomorphic to $k\mathcal H$ by \cref{prop:cleaningChainsPivot}, and again we are done.
\end{proof}


We can also prove a rough converse to \cref{thm:mainPivot} using \cref{lem:bouchetPivot} below, which is analogous to \cref{lem:bouchet3Ways} but for pivot-minors. 
A similar statement holds for non-bipartite graphs; see Dabrowski et al.~{\cite[Lemma 2.2]{DDJKKOP2023}}. The version stated below follows from \cref{lem:pivotFundMatroid} and the fact that for any matroid $M$, minor $N$ of $M$, and element $e \in E(M)\setminus E(N)$, the matroid $N$ is a minor of at least one of $M/e$ and $M\setminus e$. 
Given a bipartite graph~$\mathcal{G}$ and a vertex $v$ of $\mathcal{G}$, we define $\mathcal{G}-v$ and $\mathcal{G}/v$ in the same way as for non-bipartite graphs. 

\begin{lemma}[Dabrowski et al.~{\cite[Lemma 2.2]{DDJKKOP2023}}]
\label{lem:bouchetPivot}
    If $\mathcal{H}$ and $\mathcal{G}$ are bipartite graphs so that $\mathcal{H}$ is a pivot-minor of $\mathcal{G}$, and if $v$ is a vertex of~$\mathcal{G}$ not in~$\mathcal{H}$, then $\mathcal{H}$ is a pivot-minor of at least one of $\mathcal{G}-v$ and $\mathcal{G}/v$. 
\end{lemma}







\noindent Using \cref{lem:bouchetPivot} instead of \cref{lem:bouchet3Ways}, we can obtain the following rough converse to \cref{thm:mainPivot}. The proof follows the same approach as in \cref{prop:converse}, and we omit it. 

\begin{proposition}
\label{prop:conversePivot}
    For any bipartite graph $\mathcal{H}$ and non-negative integer $p$, if $\mathcal{G}$ is a bipartite graph with a $p$-pivot-perturbation $\widetilde{\mathcal{G}}$ with no pivot-minor strongly isomorphic to $\mathcal{H}$, then $\mathcal{G}$ has no pivot-minor strongly isomorphic to $(p+1)\mathcal{H}$.
\end{proposition}

\section{Minors of binary matroids}
\label{sec:matroids}

In this section we prove \cref{cor:binaryMatroids} and its rough converse, \cref{prop:binMatConverse}.

Let $M$ be a binary matroid. Given a non-negative integer~$p$, we say that a binary matroid $\widetilde{M}$ is a \emph{rank\nobreakdash-$p$ perturbation} of $M$ if $E(M)=E(\widetilde{M})$ and there exist matrices $A$ and $\widetilde{A}$ over~$\GF(2)$ which represent $M$ and $\widetilde{M}$, respectively, such that $A$ and $\widetilde{A}$ have the same set of row indices, and the rank of $A-\widetilde{A}$ is at most~$p$. 
We note that rank-$p$ perturbations of graphs are very different from rank-$p$ perturbations of their cycle matroids. Intuitively, this is because the first definition operates in the adjacency matrix while the second operates in the incidence matrix.

We need the bipartite Ramsey theorem.

\begin{theorem}[Bipartite Ramsey theorem; Beineke and Schwenk~\cite{BS1976}]\label{thm:bipRamsey}
    For any positive integers $k$ and $\ell$, there exists a positive integer $N=B(k,\ell)$ such that every $2$-coloring of the edges of the complete bipartite graph $K_{N,N}$ contains a monochromatic $K_{k,\ell}$ subgraph. 
\end{theorem}

We also require an alternate framework for low-rank perturbations. Let $M_1$ and $M_2$ be matroids on the same set of elements. If there exists a matroid $M$ with an element~$e$ so that 
\[ M_1=M\setminus e \text{ and }M_2=M/e,\] 
then we say that $M_1$ is an \emph{elementary lift} of $M_2$ and $M_2$ is an \emph{elementary projection} of $M_1$. 
Given binary matroids $M_1$ and $M_2$, we write $\dist(M_1, M_2)$ for the minimum number of elementary lifts and elementary projections needed to transform~$M_1$ into $M_2$, where we additionally require that the larger matroid $M$ is always binary.

 We require the following lemma, which is stated in~\cite{highlyConnMatroidsConj} without proof. A proof can be found in a paper by Grace and van Zwam~\cite[Lemma 2.6]{GZ2018}.

\begin{lemma}[Geelen, Gerards, and Whittle~{\cite[Lemma~2.1]{highlyConnMatroidsConj}}]
\label{lem:liftsProjections} Let $M_1$ and $M_2$ be binary matroids.
    If $M_1$ is a rank-$p$ perturbation of $M_2$, then $\dist(M_1, M_2)\leq 2p$. Conversely, if $\dist(M_1, M_2)\leq p$, then $M_1$ is a rank-$p$ perturbation of $M_2$.
\end{lemma}

We can now use this set-up to understand pivot-perturbations from a matroidal perspective. We only prove one direction in the following lemma since that is all we require.

\begin{lemma}
\label{lem:pivotElementary}
    Let $p$ be a non-negative integer and let $M_1$ and $M_2$ be binary matroids with respective bases $B_1$ and $B_2$ such that $\mathcal{F}(M_2, B_2)$ is a $p$-pivot-perturbation of $\mathcal{F}(M_1, B_1)$. Then $\dist(M_1, M_2)\leq p$.
\end{lemma}
\begin{proof}
    We prove the claim by induction on $p$. The base case when $p=0$ holds by \cref{lem:pivotFundMatroid}. So we may assume that $p\geq 1$ and the lemma holds for all smaller~$p$.
    
    Let $\mathcal{G}$ be a bipartite graph on at most $\abs{E(M_1)}+p$ vertices that contains both $\mathcal{F}(M_1, B_1)$ and $\mathcal{F}(M_2, B_2)$ as pivot-minors. By pivoting in $\mathcal{G}$, we may assume that $\mathcal{F}(M_1, B_1)$ is an induced subgraph of $\mathcal{G}$. 
    Fix an arbitrary vertex $v$ of $\mathcal{G}$ which is not an element of $M_1$. 
    By \cref{lem:bouchetPivot}, the bipartite graph $\mathcal{F}(M_2, B_2)$ is a pivot-minor of at least one of~$\mathcal{G}-v$ and~$\mathcal{G}/v$. 
    If $\mathcal{F}(M_2, B_2)$ is a pivot-minor of~$\mathcal{G}-v$, then $\mathcal{F}(M_2, B_2)$ is a $(p-1)$-pivot-perturbation of~$\mathcal{F}(M_1, B_1)$, and we are done by induction on~$p$. 

    So we may assume that $\mathcal{F}(M_2, B_2)$ is a pivot-minor of $\mathcal{G}/v$. Let $M$ be a binary matroid with a base $B$ such that $\mathcal{F}(M, B)$ is the induced subgraph of $\mathcal{G}$ on vertex set $E(M_1) \cup \{v\}$. By \cref{lem:pivotFundMatroid}, $M_1$ is a minor of $M$, and so either $M_1 = M/v$ or $M_1=M \setminus v$. Moreover, notice that $\mathcal{F}(M_2, B_2)$ is a $(p-1)$-pivot-perturbation of $\mathcal{F}(M, B)/v$ since they are both pivot-minors of $\mathcal{G}/v$. Thus either $\dist(M/v, M_2)\leq p-1$ or $\dist(M\setminus v, M_2)\leq p-1$. In any case, we find that $\dist(M_1, M_2)\leq p$.
\end{proof}

We are ready to prove \cref{cor:binaryMatroids} as a corollary of \cref{thm:mainPivot}. We restate the corollary below for convenience. 
\binaryMatroids*
\begin{proof}
    If $H$ has no edges, then every matroid has $M(kH)$ as a minor for every positive integer~$k$. Thus, we may assume that $H$ has an edge.
    
    If $H$ has no cycles, then any matroid of rank at least $k\abs{E(H)}$ has a minor isomorphic to~$M(kH)$. Moreover, any binary matroid of rank less than $k\abs{E(H)}$ is a rank-$(k\abs{E(H)}-1)$ perturbation of a matroid of rank~$0$. Since $H$ is acyclic and has at least one edge, this matroid does not have a minor isomorphic to $M(H)$. Hence, we may assume that $H$ contains a cycle.

    If every edge of~$H$ is a loop, then any matroid whose dual has rank at least $k\abs{E(H)}$ has a minor isomorphic to $M(kH)$. Moreover, if the dual of a binary matroid $M$ has rank less than $k\abs{E(H)}$, then $M$ is a rank-$(k\abs{E(H)}-1)$ perturbation of a matroid $U$ of rank $\abs{E(M)}$. 
    This is because if $\bigl[ I \mid A\bigr]$ is a matrix representing $M$, then 
    $\bigl[\begin{smallmatrix} I  & A \\ 0 & 0 \end{smallmatrix}\bigr]$ and 
    $\bigl[\begin{smallmatrix} I  & A \\ 0 & I \end{smallmatrix}\bigr]$ are matrices 
    representing $M$ and $U$, respectively.
    Clearly, $U$ has no minor isomorphic to $M(H)$.
    Thus, we may assume that $H$ has at least one non-loop edge.

    Suppose that $H$ has no cycles of length at least $2$. 
    Let $m_0$ be the number of loops of~$H$ and $m_1$ be the number of non-loop edges of $H$.
    By assumption, $m_0, m_1\ge 1$.
    Let $N=B(km_0+1,km_1+1)$ be the integer from the bipartite Ramsey theorem (\cref{thm:bipRamsey}).
    Let $M$ be a binary matroid. 
    Let $B$ be an arbitrary base of $M$.
    Let us fix a representation~$\bigl[ I \mid A\bigr]$ of $M$ over $\GF(2)$, where $I$ is a $B\times B$ identity matrix and $A$ is a $B\times (E(M)\setminus B)$ matrix.
    If $M$ has rank less than $N$, then $M$ is a rank-$(N-1)$ perturbation of a matroid of rank $0$, which has no minor isomorphic to $M(H)$ as $m_1\neq 0$.
    If the dual matroid of $M$ has rank less than~$N$, then $M$ is a rank-$(N-1)$ perturbation of a matroid~$U$ of rank $\abs{E(M)}$, which has no minor isomorphic to $M(H)$ as $m_0\neq 0$, because 
    $\bigl[\begin{smallmatrix} I  & A \\ 0 & 0 \end{smallmatrix}\bigr]$ and 
    $\bigl[\begin{smallmatrix} I  & A \\ 0 & I \end{smallmatrix}\bigr]$ are matrices 
    representing $M$ and~$U$, respectively.
    Therefore, we may assume that the rank of both $M$ and its dual are at least $N$.
    By~\cref{thm:bipRamsey}, there are $X\subseteq B$ and $Y\subseteq E(M)\setminus B$ such that $\abs{X}=km_0+1$ and $\abs{Y}=km_1+1$, and the $X\times Y$ submatrix $A[X,Y]$ of $A$ is the all-$1$ matrix or the zero matrix.
    Let $M'$ be the matroid obtained from $M$ by contracting~$B\setminus X$ and deleting~$E(M)\setminus (B\cup Y)$. 
    Then $\bigl[ I\mid 0 \bigr]$ or $\bigl[ I\mid 1 \bigr]$ is a representation of $M'$ over $\GF(2)$, where $I$ is an $X\times X$ identity matrix and $0$ or $1$ is the $X\times Y$ zero or all-$1$ matrix, respectively.
    Let $e\in X$ and $f\in Y$. 
    In both cases, $M$ has a minor isomorphic to~$M(kH)$; in the first case, $M/e\setminus f$ is isomorphic to~$M(kH)$, and in the second case, $M\setminus e/f$ is isomorphic to~$M(kH)$.
    Therefore, we may assume that $H$ has a cycle of length at least~$2$. 

    Let $T$ be an arbitrary spanning forest of $H$. By \cref{thm:planarBipCircle}, the fundamental graph $\mathcal{F}(M(H),E(T))$ is a bipartite circle graph. 
    Set $\mathcal{H}=\mathcal{F}(M(H),E(T))$ for convenience. Since $H$ contains a cycle of length at least $2$, the bipartite graph $\mathcal{H}$ has at least one edge. 
    Thus, by \cref{thm:mainPivot}, there exists an integer $p_1=p_1(k,\mathcal{H})$ so that each bipartite graph~$\mathcal{G}$ either has a pivot-minor strongly isomorphic to $k\mathcal{H}$, or has a $p_1$-pivot-perturbation $\widetilde{\mathcal{G}}$ which has no pivot-minor strongly isomorphic to $\mathcal{H}$.

    We now prove that the corollary holds with~$p=p_1$. Let $M$ be a binary matroid and $B$ be an arbitrary base of~$M$. Then either $\mathcal{F}(M,B)$ has a pivot-minor strongly isomorphic to $k\mathcal{H}$, or $\mathcal{F}(M,B)$ has a $p_1$-pivot-perturbation $\widetilde{\mathcal{G}}$ which has no pivot-minor strongly isomorphic to $\mathcal{H}$. In the first case, note that $k \mathcal{H}$ is strongly isomorphic to the fundamental graph of $M(kH)$ with respect to the base obtained by taking the $k$ copies of $E(T)$. Thus, by \cref{lem:pivotFundMatroid}, $M(kH)$ is isomorphic to a minor of $M$, and we are done.

    So we may assume that $\mathcal{F}(M,B)$ has a $p_1$-pivot-perturbation $\widetilde{\mathcal{G}}$ which has no pivot-minor strongly isomorphic to $\mathcal{H}$. Let $\widetilde{M}$ be a binary matroid with a base $\widetilde{B}$ so that $\widetilde{G}=\mathcal{F}(\widetilde{M},\widetilde{B})$. By \cref{lem:pivotElementary}, we have $\dist(M, \widetilde{M})\leq p_1$. So by \cref{lem:liftsProjections}, the matroid $\widetilde{M}$ is a rank-$p_1$ perturbation of $M$. Moreover, by \cref{lem:pivotFundMatroid}, $\widetilde{M}$ has no minor isomorphic to~$M(H)$, which completes the proof of \cref{cor:binaryMatroids}.
\end{proof}

In order to prove the rough converse of \cref{cor:binaryMatroids}, we now prove the following helpful lemmas.

\begin{lemma}
\label{lem:converseLifts}
    Let $k \geq 2$ be an integer, and let $M_1$ and $N$ be matroids so that $M_1$ contains a minor isomorphic to $kN$. Then every elementary projection $M_2$ of $M_1$ contains a minor isomorphic to $(k-1)N$.
\end{lemma}
\begin{proof}   
    Let $M$ be a matroid with an element $e$ so that $M_1=M\setminus e$ and $M_2=M/ e$.
    Observe that for any element~$f$ of~$M_1$, 
    $M_1\setminus f$ and $M_1/f$ are elementary projections of~$M_2\setminus f$ and~$M_2/f$, respectively.
    Thus, we may assume that $M_1=kN$.
    Let $(X_1,X_2, \ldots, X_k)$ be the partition of~$E(M_1)$ so that $M_1\setminus (E(M_1)\setminus X_i)$ is isomorphic to $N$ for each $i\in [k]$.
    Let $E=E(M)=\bigl(\bigcup_{i=1}^k X_i\bigr)\cup\{e\}$.
    Let $r$ be the rank function of $M$.
    Let $\lambda(Z)=r(Z)+r(E\setminus Z)-r(E)$ be the connectivity function of~$M$. It is well known that $\lambda$ is non-negative and submodular; $\lambda(X)+\lambda(Y)\geq \lambda(X\cup Y)+\lambda(X\cap Y)$ for all subsets~$X$ and~$Y$ of~$E$.
    
    We may assume that $e$ is not a coloop in~$M$, because otherwise $M_1=M_2$. This means that $r(E\setminus e)=r(E)$.

    Suppose that there is $i\in[k]$ such that $E\setminus (X_i\cup\{e\})$ does not span $e$ in~$M$. Then for every subset $Z$ of $E\setminus (X_i\cup \{e\})$, $Z$ does not span~$e$ in~$M$ and therefore $r(Z\cup \{e\})=r(Z)+1$. Therefore    
    $M/e\setminus X_i=M\setminus e\setminus X_i$, proving that $M_2$ contains a minor isomorphic to $(k-1)N$.

    So we may assume that $E\setminus (X_i\cup\{e\})$ spans~$e$ in~$M$ for every $i\in [k]$. 
    Since $X_i$ is a disjoint union of components of $M_1$, we have $r(X_i)+r(E\setminus (X_i\cup \{e\}))=r(E\setminus\{e\})$.
    As $E\setminus (X_i\cup\{e\})$ spans~$e$ in $M$ and $e$ is not a coloop of $M$, we have $\lambda(X_i)=0$.
    By the submodularity of $\lambda$, we have $\lambda(\bigcup_{j=1}^k X_j)=\lambda(\{e\})=0$.
    As $e$ is not a coloop of $M$, we have $r(e)=0$, and therefore $e$ is a loop of $M$, implying that $M_1=M_2$.
\end{proof}

\begin{lemma}
\label{lem:converseLiftsDual}
    Let $k \geq 2$ be an integer, and let $M_1$ and $N$ be matroids so that $M_1$ contains a minor isomorphic to $kN$. Then every elementary lift $M_2$ of $M_1$ contains a minor isomorphic to $(k-1)N$.
\end{lemma}
\begin{proof}
    Let $M$ be a matroid with an element $e$ so that $M_1=M/e$ and $M_2 = M \setminus e$. Thus $M_1^*=M^*\setminus e$ and $M_2^*=M^*/ e$, and so $M_2^*$ is an elementary projection of $M_1^*$. 
    Note that $M_1^*$ contains $(kN)^*$ as a minor, and that $(kN)^*=k(N^*)$. Thus, by \cref{lem:converseLifts}, $M_2^*$ contains $(k-1)(N^*)$ as a minor. By taking the dual, we find that $M_2$ contains $(k-1)N$ as a minor.
\end{proof}

Finally, we are ready to prove the rough converse of \cref{cor:binaryMatroids}.

\begin{proposition}
\label{prop:binMatConverse}
    Let $N$ be a matroid and $p$ be a non-negative integer. 
    If a binary matroid~$M$ has a rank-$p$ perturbation~$\widetilde{M}$ with no minor isomorphic to~$N$, then $M$ has no minor isomorphic to~$(2p+1)N$.
\end{proposition}
\begin{proof}
    By \cref{lem:liftsProjections}, it suffices to show that for any matroid $M_1$ containing a minor isomorphic to $kN$ for some integer $k \geq 2$, every elementary lift and every elementary projection of $M_1$ contains $(k-1)N$ as a minor. This is exactly \cref{lem:converseLifts,lem:converseLiftsDual}. 
\end{proof}

\section{Acknowledgments}
We would like to thank James Davies for pointing out that \cref{cor:binaryMatroids} is a corollary of \cref{thm:mainPivot}, and for letting us include this result in the paper. We would also like to acknowledge the 2025 Oberwolfach Workshop on Graph Theory and the 2025 Barbados Graph Theory Workshop held at the Bellairs Research Institute. Much of the work in this paper was performed during these workshops.

\bibliographystyle{alphaurl}
\bibliography{vm}

\end{document}